\documentclass[a4paper]{amsart}
\usepackage[top=2.2cm,left=2.8cm,right=2.8cm,bottom=2.8cm]{geometry}

\usepackage[foot]{amsaddr}
\usepackage{hyperref}
\hypersetup{
    colorlinks,
    citecolor=green,
    filecolor=black,
    linkcolor=blue,
    urlcolor=blue
}
\hypersetup{linktocpage}

\usepackage{cleveref}
\crefname{figure}{Fig.}{Figs.}
\Crefname{figure}{Fig.}{Figs.}
\usepackage{graphicx} 
\usepackage{amsmath,amsthm,amsfonts,amssymb}  
\numberwithin{equation}{section}
\usepackage[backend=biber, style=numeric, giveninits=true, doi=false, url=false, eprint=false, maxbibnames=5, minbibnames=5]{biblatex}

\usepackage{subfig} 
\usepackage{svg}
\usepackage{algpseudocode}
\usepackage{enumitem}
\usepackage{todonotes}
\usepackage{xparse}
\usepackage{tikz-cd}
\title[Grassmannian interpolation for parametric multiscale PDEs]
{Efficient localised model reduction for multiscale PDEs via Grassmannian interpolation}
\date{October 6, 2026}
\author{Christian Alber$^1$}
\address{\rm $^1$ Institute for Mathematics, Heidelberg University, 69047 Heidelberg, Germany}
\email[Christian Alber]{c.alber@uni-heidelberg.de}

\author{Markus Bachmayr$^2$}
\address{\rm $^2$ Institut f\"ur Geometrie und Praktische Mathematik, RWTH Aachen University, Im S\"usterfeld 2, 52072 Aachen, Germany}
\email[Markus Bachmayr]{bachmayr@igpm.rwth-aachen.de}

\author{Robert Scheichl$^1$}
\email[Robert Scheichl]{r.scheichl@uni-heidelberg.de}

\author{Huqing Yang$^2$}
\email[Huqing Yang]{yang@igpm.rwth-aachen.de}
\thanks{
M.B.\ and H.Y.\ acknowledge funding by Deutsche Forschungsgemeinschaft (DFG, German Research Foundation) -- project number 442047500/SFB 1481 \emph{Sparsity and Singular Structures}. } 

\renewcommand{\u}{\boldsymbol{u}}
\renewcommand{\v}{\boldsymbol{v}}
\newcommand{\w}{\boldsymbol{w}}

\newcommand{\Amax}{A_{\mathrm{max}}}
\newcommand{\Amin}{A_{\mathrm{min}}}

\renewcommand{\H}[1]{H^1(#1)}

\newcommand{\HzC}[1]{H^1_0(#1; \mathbb{C})}
\newcommand{\HzR}[1]{H^1_0(#1; \mathbb{R})}

\newcommand{\Hdir}[1]{H^{1}_{D}(#1;\R)}
\newcommand{\Linf}[1]{L^{\infty}(#1)}

\NewDocumentCommand{\linop}{m g}{%
  \IfNoValueTF{#2}
    {\mathcal{L}(#1)}
    {\mathcal{L}(#1,#2)}%
}

\newcommand{\htraceC}[1]{H^{1/2}(\partial #1; \mathbb{C})}

\newcommand{\htraceZC}[1]{H_0^{1/2}(\partial #1; \mathbb{C})}
\newcommand{\htraceZR}[1]{H_0^{1/2}(\partial #1; \mathbb{R})}
\newcommand{\trace}{\gamma}

\newcommand{\norm}[2]{\|{#1}\|_{#2}}

\newcommand{\di}{\operatorname{div}}

\newcommand{\orthProj}[1]{\mathbb{P}_{#1}}

\newcommand{\om}[1]{\omega_{#1}}
\newcommand{\omi}{\om{i}}
\newcommand{\os}[1]{\omega^{\ast}_{#1}}
\newcommand{\osi}{\os{i}}

\newcommand{\numDomPerDir}{N}
\newcommand{\dglob}{d}
\newcommand{\spatDim}{N}

\newcommand{\nlocGlob}{\boldsymbol{n}}
\newcommand{\nloc}{n}

\newcommand{\dloc}{d_{\mathrm{loc}}}

\NewDocumentCommand{\pu}{ g }{%
  \IfValueTF{#1}{P_{#1}^{\partial}}{P^{\partial}}%
}

\newcommand{\puOrig}{P}
\newcommand{\AharmEx}{E}
\newcommand{\gap}{\delta}

\newcommand{\specProj}{\Pi^{\nloc{}}}
\newcommand{\specProjR}{\mathcal{V}_{\nloc, \R}^{\partial}}
\newcommand{\specProjPuR}{\mathcal{V}_{\nloc, \R}}
\newcommand{\optSpace}{\mathcal{V}}
\newcommand{\optSpacePu}{\mathcal{V}_{\nloc}^{\partial}}
\newcommand{\optSpacePuOrig}{\mathcal{V}_{\nloc}}

\newcommand{\specProjRTan}{g^{\partial}}
\newcommand{\specProjPuRTan}{g}
\newcommand{\specProjRApprox}{\widehat{\mathcal{V}}_{\nloc, \R}^{\partial}}
\newcommand{\specProjPuRApprox}{\widehat{\mathcal{V}}_{\nloc, \R}}
\newcommand{\specProjRTanApprox}{g^{\partial}_S}
\newcommand{\specProjPuRTanApprox}{g_S}
\newcommand{\z}{\boldsymbol{z}}

\newcommand{\C}{\mathbb{C}}
\newcommand{\R}{\mathbb{R}}

\newcommand{\GrR}{\mathrm{Gr}}

\newcommand{\Yglob}{[-1,1]^{\dglob}}
\newcommand{\Yloc}{[-1,1]^{\dloc}}
\newcommand{\para}{\z}
\newcommand{\hb}{\boldsymbol{h}}

\newcommand{\compconj}[1]{%
  \overline{#1}%
}

\newtheorem{theorem}{Theorem}[section]
\newtheorem{lemma}[theorem]{Lemma}
\newtheorem{corollary}[theorem]{Corollary}
\newtheorem{proposition}[theorem]{Proposition}
\newtheorem{remark}[theorem]{Remark}
\newtheorem{assumption}[theorem]{Assumption}
\newtheorem{definition}[theorem]{Definition}

\crefname{section}{Section}{Sections}
\crefname{figure}{Fig.}{Figs.}
\crefname{theorem}{Theorem}{Theorems}
\crefname{lemma}{Lemma}{Lemmas}
\crefname{corollary}{Corollary}{Corollaries}
\crefname{proposition}{Proposition}{Propositions}
\crefname{remark}{Remark}{Remarks}
\crefname{assumption}{Assumption}{Assumptions}

\newcommand{\bw}{\mathbf{w}}
\newcommand{\balpha}{\boldsymbol{\alpha}}

\newcommand{\cI}{\mathcal{I}}

\begin{document}



 \begin{abstract}
     Multiscale, parameter-dependent partial differential equations (PDEs) pose severe
     computational challenges due to strong coefficient heterogeneity and high-dimensional
     parameter spaces. We develop a geometric interpolation approach within the multiscale
     generalized finite element method (MS-GFEM) that targets the most expensive component:
     computing parameter-dependent optimal local approximation spaces. Leveraging the
     spatial localization of MS-GFEM and assuming local parameter dependence, we decompose
     the global problem into parametrically low-dimensional local subproblems.
     The optimal
     subspaces for each parameter are identified as points on a Grassmann manifold and
     approximated via Grassmann interpolation on sparse grids, which preserves the geometric
     structure of these spaces while efficiently handling high-dimensional parameter spaces.
     The resulting localized model reduction method inherits the nearly exponential spatial
     convergence of MS-GFEM and the parametric convergence rates of sparse grids. Numerical experiments for elliptic problems confirm the theoretical
     convergence results.

 \end{abstract}

\maketitle

\section{Introduction}
\label{sec:introduction}

Many problems in science and engineering are governed by partial differential equations
(PDEs) whose coefficients exhibit strong heterogeneity across multiple spatial scales. When
these coefficients additionally depend on parameters, representing uncertain material
properties, loading conditions, or design variables, the resulting parametric multiscale
PDEs present a twofold computational challenge. On the one hand, resolving the fine-scale
structure of the coefficients requires very fine spatial discretizations. For classical finite
element methods, this is essential to achieve acceptable accuracy~\cite{Babuka1999} and leads
to very large, ill-conditioned linear systems. On the other hand, exploring the parameter
space demands repeated solves across potentially high-dimensional parameter domains.
Standard numerical methods that attempt to address both aspects simultaneously face
computational costs that grow prohibitively with spatial resolution and parametric
dimensionality.
 
Multiscale methods address the spatial challenge by incorporating the fine-scale structure of
the coefficients into problem-adapted approximation spaces, enabling accurate approximations
even on coarse grids. Some methods exploit structural assumptions on the coefficients such as
periodicity or scale separation, including the Heterogeneous Multiscale
Method~\cite{EE03}, the Two-Scale Finite Element Method~\cite{MaS02}, and the Multiscale
Finite Element Method~\cite{HoW97}. Other methods provide good approximations under minimal
assumptions on the coefficients, such as the Generalized Multiscale Finite Element
Method~\cite{EfeGH13,ChuEL18b}, Adaptive Local Bases~\cite{GraGS12}, the (Super-)Localized
Orthogonal Decomposition method~\cite{MalP14,HenP13,HaPe21b,Freese2023}, or
Gamblets~\cite{Owh17}; for comprehensive overviews see~\cite{AltHP21,ma2025unified}. These
methods have in common that they construct basis functions by solving local problems on
subdomains. Among them, the Multiscale Spectral Generalized Finite Element Method (MS-GFEM),
introduced in~\cite{BabL11} and further developed
in~\cite{babuvska2014machine,babuvska2020multiscale,ma2022novel,ma2022error,ma2025unified},
stands out for its strong localization properties. The required oversampling regions are
typically small compared to other multiscale methods. This makes the MS-GFEM particularly
well-suited for parametric problems, since on each subdomain the coefficient, and hence the
optimal local approximation space, depends on only a few parameters, reducing
the effective parametric dimension of the local problems.
 
For parametric and stochastic PDEs, a variety of approaches have been developed to
approximate the parameter-to-solution map. These include stochastic Galerkin
methods, stochastic collocation
methods~\cite{nobileAnisotropicSparseGrid2008,BabuskaNobileTempone2010}, sparse grid
techniques~\cite{BungartzGriebel2004}, tensor
decompositions~\cite{Oseledets2011,Bachmayr2023,BachmayrCohen2017,BachmayrCohenMigliorati2017a,BachmayrCohenDeVoreMigliorati2017b},
and generalized polynomial chaos expansions~\cite{XiuKarniadakis2002}. Model reduction
methods such as proper orthogonal decomposition (POD) and greedy reduced basis (RB)
methods~\cite{BennerGugercinWillcox2015} approximate the solution manifold in
low-dimensional spaces and have been combined with optimal local approximation spaces in
localized model order reduction
approaches~\cite{smetana2016optimal,Buhr2018,buhr2020localized,keil2023online,abdulle2015reduced}.
While greedy algorithms can achieve error decay rates comparable to the Kolmogorov $n$-width,
existing results typically guarantee accuracy only on finite training sets rather than
uniformly over the parameter domain~\cite{smetana2016optimal}. Moreover, in
high-dimensional parameter spaces, constructing sufficiently rich training sets becomes
prohibitively expensive, and the resulting reduced local spaces may be impractically large.
 
In this work, we propose a different strategy. Rather than building a single
local space that is accurate for all parameter values, we construct
parameter-dependent local approximation spaces by interpolating the entire map
from parameters to optimal local approximation spaces. The key observation is
that the MS-GFEM framework reduces the global parametric problem to independent
local subproblems, each depending on only a few parameters. On each subdomain,
the optimal local approximation space is determined by a local eigenvalue
problem whose leading $k$-dimensional eigenspace varies smoothly with the
parameters. Since these eigenspaces are subspaces of a fixed
Hilbert space, they are naturally identified as points on the Grassmann
manifold. To approximate the resulting parameter-to-subspace map, we compose it
with the Riemannian logarithm at a base point to obtain a map into the tangent
space, where standard sparse grid
interpolation~\cite{nobileAnisotropicSparseGrid2008} can be applied. The
interpolant is then mapped back to the Grassmann manifold via the Riemannian
exponential. This approach, similar
to~\cite{amsallem2008interpolation,savostianov2025grassmanian}, preserves the
geometric structure of the Grassmannian throughout the approximation process.
 
The starting point of our error analysis is a result that bounds the local approximation error for
any subspace of the local operator-harmonic spacein terms of the optimal local approximation
error and the largest principal angle between the given local subspace and
the optimal one. A key technical difficulty in the parametric setting
is that the operator-harmonic spaces at the interpolation points depend on the parameter, and
the interpolation of such spaces does in general not yield an operator-harmonic
space. Consequently, the local approximation error results are not directly applicable. To
overcome this, we observe that operator-harmonic functions are uniquely determined by their
boundary traces. Instead of interpolating
subspaces of the operator-harmonic space, we interpolate suitable  boundary trace
spaces and then apply operator-harmonic extension to recover the local
approximation spaces. With this reformulation, we can bound the local approximation error of the
interpolated spaces in terms of the geodesic distance on the Grassmann manifold between the
interpolated and optimal local subspaces.
 
It remains to control this distance. The error analysis consists of two main
steps: bounding the geodesic distance on the manifold by the interpolation error
in the tangent space, and bounding the latter. For the first step, we
extend a finite dimensional
result of \cite{jacobsson2025approximating} to the Grassmannian in infinite
dimensions. For the second, we establish real analyticity of the
parameter-to-tangent-space map under a spectral gap assumption on the local
eigenvalue problems (cf.~\cite{grubisic2023stochastic}), which allows us to
apply convergence results for sparse grid interpolation. Combining the MS-GFEM
spatial error estimates with the sparse grid convergence estimates on the
Grassmannian yields a method with nearly exponential convergence, both in the
spatial approximation and in the parametric interpolation error. The dimensionality of
the local interpolation problems can be controlled by choosing the subdomain
partition in MS-GFEM appropriately. 
 
The remainder of this paper is organized as follows. \Cref{sec:ms_gfem} introduces the
problem setting and recalls the construction and error analysis of the GFEM. Moreover, it 
introduces the MS-GFEM in the parametric setting and the relevant operator, whose 
eigenfunctions determine the optimal local approximation spaces of the MS-GFEM. Its
two main results are \cref{thm:T_is_holomorphic}, establishing the 
holomorphy of this operator, and
\cref{thm:real_analyticity_spectral_projection}, stating the real analyticity of the map 
from parameters to the optimal local approximation spaces. 
\Cref{sec:interpolation} reviews anisotropic sparse grid interpolation, extends the manifold
approximation result of \cite{jacobsson2025approximating} to the 
Grassmannian in infinite dimensions, and combines these to obtain local approximation error bounds for the 
interpolated local spaces in \cref{thm:local_approximation_error}. These local approximation errors are bounded by the interpolation 
error, which hinges on the regularity of the parameter-to-subspace map from
\cref{sec:ms_gfem}. 
\Cref{sec:numerical_experiments} presents numerical experiments
confirming the theoretical rates. \Cref{sec:appendix} collects auxiliary technical results.

\section{Parametric regularity of the optimal local MS-GFEM spaces}
\label{sec:ms_gfem}
In this section, we outline the spatial discretization underlying our
method and establish the parmatric regularity of the resulting local spaces. 
First, we recall that the global approximation error of the GFEM reduces to
the maximum over the local errors on the subdomains
(\cref{thm:global_error_gfem}), so that it suffices to approximate a single
optimal local space on each subdomain and for each parameter~$\para$.
Second, by slightly modifying the MS-GFEM, we identify the optimal
local approximation space as the dominant eigenspace of an operator $T_{\para}$ acting on a
parameter-independent space. Third, we show that $T$ extends holomorphically
in $\para$ and that the associated spectral projection is holomorphic,
subject to a spectral gap assumption. Restricting to real parameters, the induced
parameter-to-subspace map is then real analytic as a map into a Grassmannian.
\subsection{Problem setting} 
As a model problem, we consider the second-order elliptic diffusion equation with parameter dependent coefficients
\begin{equation}
    \label{eq:parametric_pde}
	\begin{aligned}
            -\nabla \cdot(A(x,\para) \nabla \u(x,\para)) & =f(x) \quad & & \text { in } \Omega,  \\
            \u(x,\para) & =0  \quad & & \text{ on } \partial \Omega,
	\end{aligned}
\end{equation}
for all $\para\in \Yglob$,
where $\Omega \subset \R^\spatDim$ is a bounded Lipschitz domain, 
$ \Yglob$ is the parameter domain, $f\in L^2(\Omega)$ and $A:D\to \Linf{\Omega;\C}$
is a holomorphic map defined on the open set $D\subset \C^{\dglob}$ containing $\Yglob$ with 
$A(\para)\in \Linf{\Omega;\R}$ for all real $\para\in D$.
For a given $\para \in D$, the solution of~\eqref{eq:parametric_pde} is defined in weak form
as a function $\u_{\para} \in H_0^1(\Omega;\C)$ such that
\begin{equation}
    \label{eq:weak_formulation}
     a_{\Omega, \para}(\u_{\para}, \v) := \int_{\Omega} A(x,\para) \nabla \u_{\para}(x) \cdot \compconj{\nabla \v(x)} \mathrm{d} x = \int_{\Omega} f(x) \compconj{\v(x)} \mathrm{d} x, \quad \v \in H_0^1(\Omega;\C).
\end{equation}
For subdomains $\om{}\subset\Omega$, we use the notation
$a_{\om{},\para}(\cdot,\cdot)$ to denote the restriction of the sesquilinear
form $a_{\Omega, \para}(\cdot,\cdot)$ to the integration domain
$\om{}$. For brevity, we also
write $A_{\para}(x) = A(x,\para)$ and $a_{\para} = a_{\Omega, \para}$.

Existence and uniqueness of the solutions $\u_{\para}$ of
\eqref{eq:weak_formulation} are guaranteed by the Lax-Milgram Theorem,
cf.~\cite[Theorem 2.1]{cohen2015approximation}, under the
\textbf{uniform ellipticity assumption}: there are $0
< \Amin \leq \Amax < \infty$ such that for all $x \in \Omega$ and for all $
\para \in D$
$$
\Amin \leq \mathrm{Re} (A(x, \para)), \quad |A(x, \para)| \leq \Amax\,.
$$

Although this work focuses on the
diffusion equation, the MS-GFEM has been successfully applied to a wide range of
problems, including parabolic, Helmholtz, and convection–diffusion problems; cf.
\cite{ma2025unified,ma2023wavenumber,schleuss2022optimal,holbach2025robust}. The
results of this work are expected to extend to these settings under
suitable assumptions, including holomorphic dependence of the PDE coefficients
on the parameters. We leave such extensions to future work.

\subsection{GFEM}
The spatial discretization underlying our method is based on the Generalized Finite Element
Method (GFEM). 
Unlike the classical FEM, the GFEM allows the use of problem-adapted local ansatz spaces,
which can yield significantly more accurate approximations than polynomial-based spaces. 
This flexibility is central to our approach.
In this subsection, we recall the general GFEM framework and a global error estimate that 
reduces the approximation error to local contributions.
Let \(\{\omega_i\}_{i=1}^M\) be an overlapping decomposition of the computational domain
\(\Omega\) into \(M \in \mathbb{N}\) bounded, connected Lipschitz subdomains. We assume that \(\cup_{i=1}^{M}
\omega_i = \Omega\) and that each point \( x \in \Omega\) belongs to at most \(\kappa \in
\mathbb N\) subdomains. Subordinate to this decomposition, we consider a partition of
unity consisting of nonzero functions \(\{\chi_i\}_{i=1}^M\) satisfying
\begin{equation}
	\label{eq:propPU}
	\begin{aligned}
			0\leq \chi_{i}\leq 1,\qquad
                        \operatorname{supp}(\chi_i) \subset
                        \overline{\omega_i},\qquad
                        \sum_{i=1}^{M}\chi_{i} \equiv 1, 
		\chi_{i}\in W^{1,\infty}(\Omega),\qquad  \max_{ x \in \Omega}\,|\nabla \chi_i|\leq \frac{C_\chi}{{ \delta_i}},
	\end{aligned}
\end{equation}
where $C_\chi>0$ and $\delta_i$ denotes the minimal width of the overlap $\{ x
\in \omi : \exists j\neq i \text{ s.t. }  x \in \om{j}\}$ of $\omi$.

The central ingredient of the GFEM is the choice of local approximation spaces
\(S_{n_i}(\omega_i) \subset H^1(\omega_i)\) with \(n_i \in \mathbb{N}\),
satisfying homogeneous boundary conditions on \(\partial \omega_i \cap \partial
\Omega\). A global approximation space is then obtained by gluing the local
spaces together via the partition of unity. Denoting \({\nlocGlob} = (\nloc_1,
\dots, \nloc_M) \in \mathbb{N}^M\), the global GFEM space is defined as
\begin{equation}
	\label{eq:Sn}
	S_{\nlocGlob}(\Omega) := \left\{ \sum_{i=1}^{M} \chi_i \phi_i : \, \phi_i \in S_{\nloc_i}(\omega_i) \right\} \subset H_0^1(\Omega).
\end{equation}
A key feature of the GFEM is that the approximation error of this global space
can be controlled by the maximum of the local approximation errors
(cf.~\cref{thm:global_error_gfem}), which enables us to construct and
interpolate the local spaces independently on each subdomain. In particular, we
will construct parameter-dependent local approximation spaces by solving a number 
of eigenproblems and interpolating the dominant eigenspaces.

To account for source terms, the GFEM employs particular functions. On each
subdomain, a local particular function $\u_i^p \in H_0^1(\omi)$ is constructed.
The global particular function is then glued together via the partition of unity,
leading to
\begin{equation*}
	\u^p :=  \sum_{i=1}^{M} \chi_i \u_i^p \in H^1_0(\Omega).
\end{equation*}
For brevity, we
suppress the dependence of the on the parameter \(\para\) in our notation throughout
the rest of this subsection. The GFEM solution is defined as the Galerkin
approximation in the affine space $\u^p + S_{\nlocGlob}(\Omega)$, seeking $\u^G =
\u^p + \u^s$ with $\u^s \in S_{\nlocGlob}(\Omega)$ such that
\begin{equation}
	\label{eq:weakformulationgfem}
	a(\u^s,\v) = F(\v) - a(\u^p,\v)\qquad \forall \v \in S_{\nlocGlob}(\Omega). 
\end{equation}
The following result bounds the global GFEM error in terms of the local
approximation errors. The specific version stated here, which incorporates the
partition of unity in the local estimates, is taken
from~\cite[Theorem~2.1]{ma2022novel}. For each $i$, let $\osi$ be a bounded,
connected Lipschitz oversampling domain with $\omi\subset\osi\subset\Omega$
and
\[
    \operatorname{dist}\bigl(\omi,\partial\osi\setminus\partial\Omega\bigr)>0.
\]
We assume that $\partial\osi\cap\partial\Omega$ is either empty or has
positive surface measure.
\begin{theorem}[Error estimate of GFEM]\label{thm:global_error_gfem} Consider a
function $\Psi := \u^{p} + \sum_{i=1}^{M}\chi_{i}\phi_{i}$, with $\u^p \in
H^1_0(\Omega)$ and  $ \phi_i \in S_{\nloc_i}(\omi)$ such that 
		\begin{equation}\label{eq:local_gfem_approx_error}
		\big\Vert \chi_{i}(\u-\u^{p}_{i}-\phi_{i})\big\Vert_{a,\,\omega_{i}}\leq \varepsilon_{i} { \big\Vert \u \big\Vert_{a,\,\osi}},
	\end{equation}
	for some $\varepsilon_{i}>0$, and
	$\omega_i\subset\omega_i^*\subset\Omega$. Then, the GFEM error can be bounded as
	\begin{equation}
		\|\u-\u^G\|_a \leq \|\u-\Psi\|_a
		\leq  {\sqrt{\kappa\kappa^*}\left(\max_{i=1,\ldots,M} \varepsilon_i\right)\big\Vert \u \big\Vert_{a}},
	\end{equation}
	where $\kappa^* \in \mathbb N$ is the maximum number of oversampling domains $\osi$
    that any $ x \in \Omega$ is contained in.
  \end{theorem}
 In particular, the GFEM approximation
 is the best approximation in the energy norm in the space $\u^p +
 S_{\nlocGlob}(\Omega)$ due to Galerkin orthogonality.

\subsection{MS-GFEM for parametric problems}
We now turn to the parametric setting and let the local approximation spaces
depend on $\para$. We first recall the MS-GFEM construction. The local particular 
functions $\u_i^p$ depend on $\para$ and are given by $\u_i^p = \psi_i|_{\om{i}}$, where 
$\psi_i\in \HzC{\os{i}}$ is the unique solution of the local problem 
\begin{equation}
    \label{eq:local_particular_function}
    a_{\os{i},\para}(\psi_i, \v) = F(\v), \quad \forall \v \in \HzC{\os{i}}.
\end{equation} 
The optimal local spaces in the MS-GFEM are designed to minimize the local approximation
errors in \eqref{eq:local_gfem_approx_error} and are given as the dominant eigenspaces of a
compact parameter-dependent operator. Due to the localization of
\cref{thm:global_error_gfem}, we will drop the subdomain index $i$ in the following and
consider a single subdomain $\om{}$ and its oversampling domain $\os{}$ with
$\om{}\subset\os{}\subset\Omega$. We assume that on the oversampling domain $\os{}$, the
coefficient $A(x,\para)$ only depends on a few parameters. Without loss of generality, we
assume that $A(x,\para)$ only depends on the first $\dloc$ parameters on $\os{}$, that is,
we can assume that $D\subset \C^{\dloc}$ is open and contains $\Yloc$. We denote by
$\HzR{\os{}}$ or $\HzC{\os{}}$ the space of real or complex-valued $H^1$ functions on
$\os{}$ with zero trace.

A central observation in the MS-GFEM is that due to~\eqref{eq:local_particular_function},
the difference $\u_{\para}|_{\os{}} - \psi$ lies in the space of $A(\para)$-harmonic functions on $\os{}$, 
given by 
$$
H_{A(\para)}(\os{})
:= \{ \u \in H^1(\os{};\C) : a_{\os{},\para}(\u,\v) = 0, \, \forall \v \in \HzC{\os{}} \}.
$$ 
In the MS-GFEM, the local approximation spaces $S_{\nloc}(\om{})$ are defined via the optimal 
approximation space of the partition of unity operator 
\begin{equation}
    \label{def:puOrig}
    \puOrig_{\para}: H_{A(\para),0}(\os{}) \to \HzC{\om{}}, \quad\u\mapsto \chi \u|_{\om{}},
\end{equation}
where 
\begin{equation}
    \label{def:harmonic_space}
H_{A(\para),0}(\os{}) :=
    \begin{cases}
    \{\,\u \in H_{A(\para)}(\os{}) :
    a_{\om{},\para}(\chi\u,\chi)=0\},
    &\text{if } \partial\os{}\cap\partial\Omega=\emptyset,\\[6pt]
    \{\,\u \in H_{A(\para)}(\os{}) : \u|_{\partial\os{}\cap\partial\Omega} = 0 \},
    &\text{if } \partial\os{}\cap\partial\Omega\neq\emptyset.
    \end{cases}
\end{equation}
Here, the spaces $H_{A(\para),0}(\os{})$ and $\HzR{\om{}}$ are equipped with the energy 
inner products $a_{\os{},\para}(\cdot,\cdot)$ and $a_{\om{},\para}(\cdot,\cdot)$,
respectively.
We denote by $\optSpacePuOrig(\para)$ the space spanned by the 
eigenfunctions corresponding to the $\nloc$ largest eigenvalues of the 
eigenproblem, seeking $(\lambda, \u) \in \mathbb{R}
 \times H_{A(\para),0}(\os{})$ such that
\begin{equation}
    \label{eq:eigenproblem_puOrig}
    a_{\om{},\para}\bigl(\puOrig_{\para}\u,\, \puOrig_{\para}\v\bigr)
    = \lambda\, a_{\os{},\para}(\u, \v),\quad \forall \v \in H_{A(\para),0}(\os{}).
\end{equation}
The space $\optSpacePuOrig(\para)$ is the optimal space in the sense that it minimizes the 
Kolmogorov $\nloc$-width of the operator $\puOrig_{\para}$. 
The local approximation space on $\om{}$ of the MS-GFEM is defined as
\begin{equation*}
    S_{\nloc}(\om{}) = 
    \begin{cases}
    \optSpacePuOrig(\para)|_{\om{}} + \C,
    &\text{if } \partial\os{}\cap\partial\Omega=\emptyset,\\[6pt]
    \optSpacePuOrig(\para)|_{\om{}},
    &\text{if } \partial\os{}\cap\partial\Omega\neq\emptyset.
    \end{cases}
\end{equation*}
Note that in the case of
interior subdomains the dimension of $S_{\nloc}(\om{})$ is $\nloc +
1$.

By \cite[Lemma 2.19, Theorem 3.8]{ma2025unified} this
choice of $S_{\nloc}(\om{})$ leads to the nearly exponentially decaying local
approximation errors in \cref{thm:global_error_gfem}, in
particular,~\eqref{eq:local_gfem_approx_error} holds with $\epsilon =
\lambda_{\nloc + 1}^{1/2} \leq C \exp(-c\nloc^{1/\spatDim})$ with some constants $c,C>0$.
To analyze the spaces $S_{\nloc}(\om{})$ in the parametric setting, we need to analyze the
family of operators $\puOrig_{\para}$ as $\para$ varies in $\Yloc$. The technical problem is that 
the domain of $\puOrig_{\para}$ changes with $\para$. To circumvent this issue, 
it is more convenient to consider a related operator on the parameter-independent space 
\begin{align*}
\htraceZC{\os{}} :=
\begin{cases}
\left\{\,\u^\partial \in H^{1/2}(\partial\os{};\mathbb{C}) : \displaystyle\int_{\partial\os{}} \u^\partial\,\mathrm{d}s = 0\,\right\}, 
&\text{if } \partial\os{}\cap\partial\Omega=\emptyset,\\[6pt]
\{\,\u^\partial \in H^{1/2}(\partial\os{};\mathbb{C}) : \u^\partial = 0 \text{ on } \partial\os{}\cap\partial\Omega\,\},
&\text{otherwise.}
\end{cases}
\end{align*}
Now, for $\para \in D$, we consider the sesquilinear forms
\begin{equation}
    \label{eq:def_sesquilinear_forms}
    a_{\os{},\para}^{\partial}(\u^{\partial}, \v^{\partial}) 
    := a_{\os{},\para}(\AharmEx_{\para}(\u^{\partial}), \AharmEx_{\compconj{\para}}(\v^{\partial})),
    \qquad \text{and} \qquad
    a_{\om{},\para}(\u, \v),
\end{equation}
on $\htraceZC{\os{}}$ and $\HzC{\om{}}$, respectively,
where 
$$
\AharmEx_{\para}: \htraceZC{\os{}} \to
H_{A(\para)}(\os{})
$$ is the $A(\para)$-harmonic extension operator: For any  $\u\in
\htraceZC{\os{}}$ we denote by $\AharmEx_{\para}(\u)$ its
$A(\para)$-harmonic extension to $\omega^*$ to $\omega^*$, defined as the unique solution of 
\begin{equation*}
 a_{\os{},\para}(\AharmEx_{\para}(\u),\v) = 0, \quad\forall \v\in \HzC{\os{}}
\end{equation*}
with boundary data on $\partial\os{}$ given by $\u$. 
For an interior domain, we define
\begin{equation}
    \label{eq:def_M_para}
    \mathcal{M}_{\para} : H^1(\os{};\C) \to \C, \quad
    \mathcal{M}_{\para}(\u)
    :=\frac{a_{\om{},\para}(\chi\u,\chi)}
    {a_{\om{},\para}(\chi,\chi)}
\end{equation}
and let
\begin{equation}
    \label{eq:def_AharmEx_zero}
    \AharmEx_{\para}^{0}\u^{\partial}
    :=\AharmEx_{\para}\u^{\partial}
    -\mathcal{M}_{\para}(\AharmEx_{\para}\u^{\partial}).
\end{equation}
For a boundary domain, let $\AharmEx_{\para}^{0}:=\AharmEx_{\para}$. Since
$\chi$ is nonzero and supported in $\overline{\om{}}$, we have
$\|\nabla\chi\|_{L^2(\om{})}>0$. Uniform ellipticity therefore ensures that the
denominator in~\eqref{eq:def_M_para} does not vanish. Moreover, for an interior
domain,
\begin{equation}
    \label{eq:AharmEx_zero_constraint}
    a_{\om{},\para}(\chi\AharmEx_{\para}^{0}\u^{\partial},\chi)
    =a_{\om{},\para}(\chi\AharmEx_{\para}\u^{\partial},\chi)
    -\mathcal{M}_{\para}(\AharmEx_{\para}\u^{\partial})
    a_{\om{},\para}(\chi,\chi)=0.
\end{equation}
Thus, since $\AharmEx_{\para}^{0}\u^{\partial}-\AharmEx_{\para}\u^{\partial}$
is constant, for any $\para\in D$ and $\u^{\partial},\v^{\partial}\in\htraceZC{\os{}}$,
we can replace $\AharmEx_{\para}$ by $\AharmEx_{\para}^{0}$ in~\eqref{eq:def_sesquilinear_forms} and obtain
\begin{equation}
    \label{eq:AharmEx_zero_energy_identity}
    a_{\os{},\para}
    (\AharmEx_{\para}^{0}\u^{\partial},
    \AharmEx_{\compconj{\para}}^{0}\v^{\partial})
    =a_{\os{},\para}^{\partial}(\u^{\partial},\v^{\partial}).
\end{equation}
\begin{lemma}
    \label{prop:inner_product_for_real}
    For real parameters $\para \in \Yloc$, the spaces $\htraceZC{\os{}}$ and $\HzC{\om{}}$ equipped with the sesquilinear forms defined in~\eqref{eq:def_sesquilinear_forms} are Hilbert spaces.
\end{lemma}
The proof is given in \cref{proof:inner_product_for_real}.
For complex parameters, on the other hand, these sesquilinear forms
are in general not Hermitian and thus do not define inner products.
\begin{lemma}
    \label{lem:AharmEx_zero_isometry}
    For real $\para\in D$, the map
    \begin{equation*}
        \AharmEx_{\para}^{0}:
        \bigl(\htraceZC{\os{}},a_{\os{},\para}^{\partial}\bigr)
        \to
        \bigl(H_{A(\para),0}(\os{}),a_{\os{},\para}\bigr)
    \end{equation*}
    is an isometry. For a boundary domain, its inverse is the trace
    operator $\trace : \HzC{\os{}} \to \htraceZC{\os{}}$. For an interior domain, its inverse is
    \begin{equation}
        \label{eq:AharmEx_zero_inverse}
        \u\mapsto \trace\u
        -\frac{1}{|\partial\os{}|}
        \int_{\partial\os{}}\trace\u\,\mathrm{d}s.
    \end{equation}
    
\end{lemma}
\begin{proof}
    Let $\os{}$ be a boundary domain. The trace operator $\trace$ is a left
    inverse of $\AharmEx_{\para}$ by definition. Moreover, for $\u \in
    H_{A(\para),0}(\os{})$, $\u$ is
    $A(\para)$-harmonic with boundary values $\trace\u$, so that 
    $\AharmEx_{\para}\trace\u = \u$. 

    Let $\os{}$ be an interior domain. Let us define $\ell : \htraceC{\os{}} \to
    \C$ by $\ell(\u^{\partial}) :=
    |\partial\os{}|^{-1}\int_{\partial\os{}}\u^{\partial}\,\mathrm{d}s$. Moreover, 
    let $\widetilde{\trace}: H^1(\os{};\C) \to \htraceC{\os{}}$ be defined by 
    $\widetilde{\trace}\u := \trace\u - \ell(\trace\u)$. 
    Let $\u^{\partial}\in\htraceZC{\os{}}$. Since $\widetilde{\trace}$ maps 
    constants to zero, we have
    \begin{align*}
        \widetilde{\trace}(\AharmEx_{\para}^{0}\u^{\partial})
        &=\widetilde{\trace}(\AharmEx_{\para}\u^{\partial})
        = \trace(\AharmEx_{\para}\u^{\partial}) 
        - \ell(\trace(\AharmEx_{\para}\u^{\partial}))
        = \u^{\partial} - \ell(\u^{\partial}) = \u^{\partial}.
    \end{align*}
    Conversely, for $\u\in H_{A(\para),0}(\os{})$, 
    we have 
    \begin{align*}
        \AharmEx_{\para}(\widetilde{\trace}\u)
        &= \AharmEx_{\para}(\trace\u - \ell(\trace\u))
        = \u - \ell(\trace\u)
    \end{align*}
    and thus 
    \begin{align*}
        \AharmEx_{\para}^0(\widetilde{\trace}\u)
        = \u - \ell(\trace\u) - \mathcal{M}_{\para}(\u - \ell(\trace\u))
        = \u - \ell(\trace\u) - \mathcal{M}_{\para}(\u) + \ell(\trace\u) 
        = \u - \mathcal{M}_{\para}(\u)
        = \u.
    \end{align*}
    Thus,
    $\AharmEx_{\para}^{0}$ is a bijection. 
    The isometry follows from~\eqref{eq:AharmEx_zero_energy_identity}.
\end{proof}
\begin{remark}
    The reason for complex conjugation in the definition of  $a_{\os{},\para}^{\partial}(\cdot,
    \cdot)$ in~\eqref{eq:def_sesquilinear_forms} 
    will become clear in the proof of \cref{prop:holomorphy_G_V}, 
    where the operator associated to the sesquilinear form
    $a_{\os{},\para}^{\partial}(\cdot, \cdot)$
    is shown to be holomorphic in $\para$.
\end{remark}
We now define the alternative parametric partition of unity operator
\begin{equation*}
    \pu: D \to \linop{\htraceZC{\os{}}}{\HzC{\om{}}}, \quad \para\mapsto \pu{\para},
\end{equation*}
where $\pu{\para}\u^{\partial} := \chi \AharmEx_{\para}^{0}(\u^{\partial})$. In
contrast to $\puOrig_{\para}$ defined in~\eqref{def:puOrig}, the domain of 
$\pu_{\para}$ is independent of $\para$, and
$\pu{\para}=\puOrig_{\para}\AharmEx_{\para}^{0}$.
\begin{lemma}
    \label{lem:compactness_pu}
    For real $\para\in D$, the operator $\pu{\para}$ is compact.
\end{lemma}
\begin{proof}
    Let $(\u^{\partial}_n)_{n\in \mathbb{N}}$ be a bounded sequence in $\htraceZC{\os{}}$. 
    Then, by \cref{lem:AharmEx_zero_isometry}, the sequence
    $(\AharmEx_{\para}^{0}\u^{\partial}_n)_{n\in \mathbb{N}}$ is a bounded sequence in
    $H_{A(\para),0}(\os{})$. Since $\puOrig_{\para}$ is compact, the sequence 
    $(\puOrig_{\para}(\AharmEx_{\para}^{0}\u^{\partial}_n))_{n\in \mathbb{N}}$ has a convergent
    subsequence. Thus, $\pu{\para}$ is compact.
\end{proof}
For real $\para\in D$, we denote by
$ \optSpacePu(\para)$ the space spanned by the $\nloc$
eigenfunctions  corresponding to the $\nloc$ largest eigenvalues
of the eigenproblem seeking $(\lambda^{\partial}, \u^{\partial}) \in \R \times \htraceZC{\os{}}$ such that
\begin{equation}
    \label{eq:eigenproblem_pu}
    a_{\om{},\para}\bigl(\pu_{\para}\u^{\partial},\, \pu_{\para}\v^{\partial}\bigr)
    = \lambda^{\partial}\, a_{\os{},\para}^{\partial}(\u^{\partial}, \v^{\partial}),\quad \forall \v^{\partial} \in \htraceZC{\os{}}.
\end{equation}
The optimal approximation space of $\pu_{\para}$ is given by $\pu_{\para}(\optSpacePu(\para))$.
The following proposition 
shows that the two optimal approximation spaces $\optSpacePu(\para)$ 
and $\optSpacePuOrig(\para)$ are related through the operator harmonic
extension operator
$\AharmEx_{\para}^{0}$.
\begin{proposition}
    \label{prop:optimal_spaces_pu_and_puOrig}
    For real $\para\in D$, the spaces $\optSpacePu(\para)$ and $\optSpacePuOrig(\para)$ are related by
    \begin{equation*}
        \optSpacePuOrig(\para) = \AharmEx_{\para}^{0}(\optSpacePu(\para)).
    \end{equation*}
\end{proposition} 
\begin{proof}
    This follows directly from \cref{lem:AharmEx_zero_isometry} and the 
    definition of $\pu_{\para}$.
\end{proof}
For every $V\subset\htraceZR{\os{}}$, real $\para\in D$, and an interior domain,
\eqref{eq:def_AharmEx_zero} gives
\begin{equation}
    \label{eq:equal_local_spaces_AharmEx_zero}
    \AharmEx_{\para}^{0}V|_{\om{}}+\R
    =\AharmEx_{\para}V|_{\om{}}+\R.
\end{equation}
This proposition lets us work with the operator $\pu_{\para}$ on parameter-independent
spaces to establish our theoretical results, while in practice the optimal approximation 
spaces can still be computed as in MS-GFEM using~\eqref{eq:eigenproblem_puOrig}.
For real $\para\in \Yloc$, we define the operator
\begin{equation}
    \label{eq:def_operator_T}
    T_{\para} := (\pu{\para})^{\ast}\pu{\para},
\end{equation}
where $(\pu{\para})^{\ast} \in \linop{\HzC{\om{}}}{\htraceZC{\os{}}}$ denotes the Hilbert adjoint of 
$\pu{\para}$ with respect to the inner products $a_{\os{},\para}^{\partial}(\cdot, \cdot)$ 
and $a_{\om{},\para}(\cdot, \cdot)$.
\subsection{Holomorphy of the operator $T$ and of the spectral projection}
\label{sec:T_holomorphy}
Holomorphic dependence of $T$ on the parameter $\para$ is the first
of two ingredients to prove the regularity of the
parameter-to-subspace map.
Holomorphy of the spectral projection then follows from a spectral gap assumption.
\begin{theorem}
    \label{thm:T_is_holomorphic}
    The operator $T: \Yloc \to \linop{\htraceZC{\os{}}}$ 
    in~\eqref{eq:def_operator_T} admits a holomorphic extension 
    $T: \widetilde{D} \to \linop{\htraceZC{\os{}}}$ to an open connected neighborhood 
    $\widetilde{D} \subset D \subset \C^{\dloc}$ of $\Yloc$.
\end{theorem}
The
proof decomposes $T$ into simpler operators and is given in \Cref{sec:proof_holomorphy_T}.
To show that the eigenspaces
of $T$ depend holomorphically on the parameters requires an assumption on the spectral
structure of the MS-GFEM eigenproblem, which we will make for the rest of the paper.
\begin{assumption}
    \label{assumption:spectral_gap}
    We assume that the spectrum of $T$ is separated into two parts, that is,
    there exist a constant $\gap>0$ and an integer $\nloc\geq1$ such that
    $$
        \inf_{\para\in\Yloc}\bigl(\lambda_{\nloc}(\para)-\lambda_{\nloc+1}(\para)\bigr) \geq \gap.
    $$
    Here, since for all $\para\in\Yloc$, the operator $T_{\para}$ is compact by \cref{lem:compactness_pu}
    and self-adjoint with respect to the inner product $a_{\os{},\para}^{\partial}(\cdot,\cdot)$, we order the eigenvalues in decreasing order
    (counted with multiplicity and completed by zeros if necessary) and set
    $$
        \sigma_1(\para) = \{\lambda_1(\para), \ldots, \lambda_{\nloc}(\para)\}, \quad
        \sigma_2(\para) = \sigma(T_{\para})\setminus\sigma_1(\para).
    $$
\end{assumption}
We first define the spectral projection $\specProj: \Yloc \to \linop{\htraceZC{\os{}}}$ by
\begin{equation}
    \label{eq:def_spectral_projection}
    \specProj_{\para} := \frac{1}{2\pi i}\int_{\Lambda(\para)}(\lambda - T_{\para})^{-1} d\lambda,
\end{equation}
where $\Lambda(\para)$ is a closed curve in the resolvent set 
of $T_{\para}$ separating $\sigma_1(\para)$ and $\sigma_2(\para)$ and enclosing $\sigma_1(\para)$ in the positive orientation.

For fixed $\para$ the spectral projection $\specProj_{\para}$ is independent of the choice of 
$\Lambda(\para)$, see \cite[Lemma 6.1]{hislop2012introduction}. 
Moreover, due to \cite[Theorem IV.3.16]{kato2013perturbation}, at any $\widetilde{\para} \in \Yloc$,
there exists a complex ball $U\subset\widetilde{D}$ centered at $\widetilde{\para}$ such that
the same curve $\Lambda$ separates the spectrum of $T_{\para}$ into the continued clusters $\sigma_1$ and $\sigma_2$
for all $\para \in U$. 
Thus, locally, the curve can be chosen independently of the parameter $\para$. 
For sufficiently small balls, the local holomorphic Riesz projections agree
at real parameters, where they select the dominant eigenspace. Since overlaps
are connected and contain open real subsets, the identity theorem gives
agreement throughout. The union of the balls, again denoted by $\widetilde{D}$,
therefore carries a well-defined rank-$\nloc$ projection $\specProj$. Also for complex
parameters, $\sigma_2(\para)=\sigma(T_{\para})\setminus\sigma_1(\para)$.
\begin{lemma}
    \label{prop:spectral_projection_properties}
    The spectral projection has the following properties:
    \begin{enumerate}
        \item[(i)] $\specProj_{\para}$ is a projection, that is $(\specProj_{\para})^2 = \specProj_{\para}$.
        \item[(ii)] For all real $\para\in \widetilde{D}$,
        $\specProj_{\para}$ is self-adjoint with respect to the inner product
        $a_{\os{},\para}^{\partial}(\cdot,\cdot)$.
        \item[(iii)] For all real $\para\in \widetilde{D}$, $\mathrm{range}(\specProj_{\para}) = \optSpacePu(\para)$. 
        \item[(iv)] For all real $\para\in \widetilde{D}$, $\specProj_{\para}$ maps
                    $\htraceZR{\os{}}$ into itself.
        
    \end{enumerate}
\end{lemma}
\begin{proof}
        (i)  Similar to \cite[Proposition 6.3]{hislop2012introduction}.\\[0.5em]
        (ii) For real $\para$, the operator
        $T_{\para}$ is self-adjoint by definition. Similar to \cite[Proposition 6.3
        (3)]{hislop2012introduction}, one can conclude that $\specProj(\para)$ is self-adjoint.\\[0.5em]
        (iii) By \cite[Proposition 6.9]{hislop2012introduction}  
        $T_{\para}$ commutes with $\specProj_{\para}$. Thus, given an eigenpair 
        $(\lambda, \u^{\partial}) \in \R \times \htraceZC{\os{}}$ of $T_{\para}$, we have
        that $(\lambda, \specProj_{\para} \u^{\partial})$ and $(\lambda,
        (I-\specProj_{\para}) \u^{\partial})$ are also eigenpairs of $T_{\para}$. Moreover,
        by \cite[Proposition 6.9]{hislop2012introduction} the spectrum 
        of $T_{\para}\specProj_{\para}$ is $\sigma_1(\para)$ and the spectrum of
        $T_{\para}(I-\specProj_{\para})$ is $\sigma_2(\para)$. Thus, either 
        $\specProj_{\para} \u^{\partial} = 0$ or $(I-\specProj_{\para}) \u^{\partial} = 0$,
        which directly implies $\mathrm{range}(\specProj_{\para}) = \optSpacePu(\para)$.\\[0.5em]
        (iv) By definition, 
        $\pu_{\para}$ maps $\htraceZR{\os{}}$ into $\HzR{\om{}}$ for real $\para$.
        To see that $(\pu_{\para})^*$ maps $\HzR{\om{}}$ into $\htraceZR{\os{}}$, we let 
        $\u \in \HzR{\om{}}$. Then, for all $\v^{\partial} \in \htraceZR{\os{}}$, we have
        \begin{equation*}
            a^{\partial}_{\os{},\para}(\mathrm{Re}((\pu_{\para})^* \u), \v^{\partial}) + i\, a^{\partial}_{\os{},\para}(\mathrm{Im}((\pu_{\para})^* \u), \v^{\partial})
            = a^{\partial}_{\os{},\para}((\pu_{\para})^* \u, \v^{\partial})
            = a_{\om{},\para}(\u, \pu_{\para} \v^{\partial}) \in \R,
        \end{equation*}
        since $\u$, $\pu_{\para} \v^{\partial}$, and $A(\para)$ are real-valued. 
        Comparing imaginary parts gives $a^{\partial}_{\os{},\para}(\mathrm{Im}((\pu_{\para})^* \u), \v^{\partial}) = 0$ 
        for all $\v^{\partial} \in \htraceZR{\os{}}$. Since $a^{\partial}_{\os{},\para}$ is an 
        inner product on $\htraceZR{\os{}}$, we get $\mathrm{Im}((\pu_{\para})^* \u) =
        0$.
        
        Now, since $T_{\para}$ is self-adjoint and maps $\htraceZR{\os{}}$ into
        itself, the orthonormal eigenfunctions of $T_{\para}$ spanning
        $\mathrm{range}(\specProj_{\para})$ can be chosen real valued. We denote them by
        $\u^\partial_1, \ldots, \u^{\partial}_{\nloc{}}$. 
        Then, for $\u^{\partial}\in \htraceZR{\os{}}$, we have $\specProj_{\para} \u^{\partial} = \sum_{i=1}^{\nloc{}} a^{\partial}_{\os{},\para}(\u^{\partial}, \u^{\partial}_i) \u^{\partial}_i \in \htraceZR{\os{}}$.
\end{proof}

\begin{proposition}
    \label{prop:holomorphy_spectral_projection}
    The spectral projection
    $
        \specProj : \widetilde{D} \to \linop{\htraceZC{\os{}}}
    $
    is holomorphic.
\end{proposition}
\begin{proof}
    By Assumption \ref{assumption:spectral_gap} and \cite[VII Theorem 1.7]{kato2013perturbation}, the mapping $\specProj$ is holomorphic in each 
    variable separately. 
    Since $\specProj$ is holomorphic in each variable separately, we have that
    $\varphi\circ \specProj$ is holomorphic in each variable separately for each
    $\C$-linear functional $\varphi:\linop{\htraceZC{\os{}}}\to\C$. Thus, by
    \cite[Theorem 2.2.8]{hormander1973introduction} the map $\varphi\circ \specProj$ is
    holomorphic for all such $\varphi$, hence $\specProj$ is
    weakly holomorphic. Using \cite[Theorem 4]{barletta2009vector}, we conclude that $\specProj$ is holomorphic.
\end{proof}

\subsection{The parameter-to-subspace map on the Grassmannian}
We have shown that the spectral projection $\specProj$ depends holomorphically on
$\para$. We now view the optimal local space as a point on a Grassmannian
and use this holomorphy to prove that the parameter-to-subspace map is real
analytic, which is the regularity required for the interpolation in
\Cref{sec:interpolation}. 

Following \cite[Section 2.2]{biliotti2014properly} 
we first recall the relevant manifolds.
Let $V$ be
a Hilbert space with inner product $(\cdot, \cdot)_V$, and let $k \in \mathbb{N}$. The space
$\linop{\R^k}{V}$ of bounded linear operators from $\R^k$ to $V$ is a Hilbert space with the inner product
\begin{equation*}
    (L_1, L_2)_{\linop{\R^k}{V}} := \mathrm{tr}(L_1^* L_2), \quad L_1, L_2 \in \linop{\R^k}{V},
\end{equation*}
where $\mathrm{tr}$ denotes the trace. 
The Stiefel manifold $\mathrm{St}_k(V)$ is the subset of $\linop{\R^k}{V}$ 
defined by 
\begin{equation*}
    \mathrm{St}_k(V) := \{ L \in \linop{\R^k}{V} : L^* L = I_k \},
\end{equation*}where $I_k$ is the identity operator on $\R^k$. The Stiefel 
manifold is a submanifold of $\linop{\R^k}{V}$ and hence inherits a Riemannian structure
from $\linop{\R^k}{V}$. The tangent space of $\mathrm{St}_k(V)$ at $L \in \mathrm{St}_k(V)$ is given by
\begin{equation*}
    T_L \mathrm{St}_k(V) = \{ X \in \linop{\R^k}{V} : L^* X + X^* L = 0 \}.
\end{equation*}
Moreover, the orthogonal group $\mathcal{O}(k)$ acts 
freely and isometrically on $\mathrm{St}_k(V)$ by right composition, and the quotient 
$\GrR_k(V) :=  
\mathrm{St}_k(V)/\mathcal{O}(k)$ is a Riemannian manifold with the quotient
metric. 
It is called the Grassmannian of $k$-dimensional subspaces of $V$.
The canonical projection $\pi: \mathrm{St}_k(V) \to \GrR_k(V)$ is a
Riemannian submersion. 
\begin{remark}
    \label{remark:analytic_manifolds}
$\linop{\R^k}{V}$, $\mathrm{St}_k(V)$ and $\GrR_k(V)$ are all real analytic
Riemannian manifolds, because the results used in their construction, in
particular the preimage theorem \cite[Theorem 4.J]{zeidler1986_nonlinear} and
the quotient manifold construction \cite[III.1.5. Proposition
10]{bourbaki1989lie}, have analytic analogues. By naturality of the Riemannian
metrics, and since the action of the orthogonal group on the Stiefel manifold is
analytic, the Riemannian metrics on the Grassmannian and the Stiefel manifold
are analytic.
\end{remark}
Let $\para_0$ be a fixed parameter in $\Yloc$. We choose $V=\htraceZR{\os{}}$
and endow it with the scalar product
$a^{\partial}_{\para_0,\os{}}(\cdot,\cdot)$. Let $\GrR_k(\htraceZR{\os{}})$ be
the Grassmannian of $k$-dimensional subspaces of $\htraceZR{\os{}}$. Note that
the construction of the Riemannian metric on $\GrR_k(\htraceZR{\os{}})$ depends
on the above chosen fixed scalar product. Let us denote by $D_{\R}$ the open (in
$\R^{\dloc}$) and connected component of $\widetilde{D}\cap \R^{\dloc}$ that
contains $\Yloc$. We define the map 
\begin{equation*}
    \specProjR:= \mathrm{range} \circ \specProj|_{D_{\R}}: 
    D_{\R} \to \GrR_{\nloc}(\htraceZR{\os{}}),
\end{equation*}
which satisfies $\specProjR(\para) = \optSpacePu(\para)\cap \htraceZR{\os{}}$ 
for all $\para\in D_{\R}$.

Moreover, we define 
\begin{align*}
    \Hdir{\os{}}
    := 
    \{\,\u \in H^1(\os{};\R) : \u|_{\partial\os{}\cap\partial\Omega} = 0 \}
\end{align*}
endowed with the standard $H^1(\os{})$ scalar product.
Let
\begin{equation*}
    \specProjPuR:D_{\R}\to
    \GrR_{\nloc}\bigl(\Hdir{\os{}}\bigr),
    \qquad
    \specProjPuR(\para)
    :=\AharmEx_{\para}^{0}\bigl(\specProjR(\para)\bigr),
\end{equation*}
where $\Hdir{\os{}}$ is equipped with its standard scalar product. By
\cref{prop:optimal_spaces_pu_and_puOrig} and because $A(\para)$ is real for
$\para\in D_{\R}$,
\begin{equation*}
    \specProjPuR(\para)
    =\optSpacePuOrig(\para)\cap H^1(\os{};\R),
    \qquad \para\in D_{\R}.
\end{equation*}
\begin{theorem}
\label{thm:real_analyticity_spectral_projection}
Let Assumption \ref{assumption:spectral_gap} hold. 
Then $\specProjR$ and $\specProjPuR$ are real analytic.
\end{theorem}
\begin{proof}
To see that $\specProjR$ is analytic, we use the real analogue of \cite[Lemma
2.1]{martin2022holomorphic}, which states that $\specProjR$ is real analytic if and 
only if $\specProjR[D_{\R}]:= \{ (\para, \specProjR(\para)) : \para \in D_{\R}\}$ is a real analytic subbundle of $D_{\R} \times
\htraceZR{\os{}}$. This is equivalent to the existence of a local real analytic
frame of $\specProjR[D_{\R}]$. It can be obtained by applying the real analytic spectral
projection $\specProj$ to a fixed reference basis.

Fix any basis $(\u^{\partial}_1, \ldots, \u^{\partial}_{\nloc})$ of $\specProjR(\para_0)$ and define
$\v^{\partial}_i(\para) := \specProj_{\para} \u^{\partial}_i$ for $i = 1, \ldots, \nloc$.
Each $\v^{\partial}_i$ is real analytic since $\para \mapsto  \specProj_{\para}$ is 
real analytic. At $\para_0$, we have $\v^{\partial}_i(\para_0) = \specProj_{\para_0} \u^{\partial}_i = \u^{\partial}_i$ 
since $\u^{\partial}_i \in \specProjR(\para_0)$.
The Gram determinant $\det\big(a^{\partial}_{\para_0,\os{}}(\v^{\partial}_i(\para), \v^{\partial}_j(\para))\big)_{i,j=1}^{\nloc}$ is 
real analytic and nonzero at $\para_0$, hence nonzero in a neighborhood $U$ 
of $\para_0$. This proves linear independence on $U$.
Finally, $\v^{\partial}_i(\para) = \specProj_{\para} \u^{\partial}_i \in \mathrm{range}(\specProj_{\para}) 
= \specProjR(\para)$, and since we have $\nloc$ linearly independent vectors in 
the $\nloc$-dimensional space $\specProjR(\para)$, they form a basis.

The proof that $\specProjPuR$ is real analytic is analogous, using the fact that 
the operator $\AharmEx_{\para}^{0}$ is real analytic in $\para$ by \cref{prop:partition_of_unity_holomorphic}.
\end{proof}

\section{Sparse grid Grassmannian interpolation of the optimal local spaces}
\label{sec:interpolation}
In \cref{thm:real_analyticity_spectral_projection}, we established that the maps
$\specProjR$ and $\specProjPuR$ are real analytic. To fix ideas for
the Grassmannian interpolation, we outline the
construction for $\specProjR$. The outline is informal, and the conditions under
which each step is well defined are stated in the analysis below. Since the
evaluation of this map involves solving an expensive eigenproblem, we seek an
efficient approximation. To achieve this, we follow the ideas in
\cite{amsallem2008interpolation} and approximate the map
$\specProjR: D_{\R} \to M := \GrR_{\nloc}(\htraceZR{\os{}})$
via interpolation on the Grassmannian $M$. Since the Grassmann
manifold $M$ does not
have a linear structure, we cannot directly apply standard interpolation
techniques. Instead, we interpolate the manifold-valued functions $\specProjR$ 
after composing them with the Riemannian logarithm
at a base point $p\in M$, resulting in a function $\specProjRTan = \log_p
\circ \specProjR: \Yloc \to T_p M$ with values in the tangent space $T_p M$,
which is a Hilbert space. This composition requires $\specProjR(\para)$ to lie
in the domain of $\log_p$ for all $\para \in \Yloc$, which we make precise in
\cref{prop:exp_diffeo_and_radial_geodesics_minimal}. Since the codomain of
$\specProjRTan$ is a linear space, standard interpolation techniques can be
applied. The resulting interpolated function~$\specProjRTanApprox$ is then
mapped back via the Riemannian exponential map, yielding the interpolated
function $\specProjRApprox = \exp_p \circ \specProjRTanApprox: \Yloc \to M$. 

The
roadmap for the analysis is as follows: First, we recall results on sparse grid
interpolation and state error estimates for functions admitting a holomorphic
extension. Second, we combine these with the regularity established in
\cref{sec:ms_gfem} to bound the interpolation error of $\specProjRTan$ in the
tangent space, and transfer this bound to an error for $\specProjRApprox$ on the
Grassmannian by extending the results of \cite{jacobsson2025approximating} to
the infinite-dimensional setting. Finally, we bound the local approximation
error of the interpolated space in terms of the Grassmannian interpolation
error.
\subsection{Sparse grid interpolation}
\label{sec:sparse_grid_interpolation}

Let $V$ be a Banach space. We describe the construction of an anisotropic sparse
grid interpolation operator for $V$-valued functions. We begin by recalling
interpolation based on Lagrange polynomials in one dimension. Let $\{z_k\}_{k
\geq 0} \subset [-1,1]$ be a sequence of distinct points for Lagrange
interpolation on $[-1,1]$. For any univariate function $g \in C^0([-1,1];V)$, we
introduce a sequence of univariate Lagrange interpolation operators $I_k:
C^0([-1,1];V) \to \mathcal{P}_k([-1,1];V)$, such that
\begin{equation}
    \label{eq:univariate_interpolation}
    I_k(g)(z) := \sum_{j=0}^{k} g(z_j) \ell_j^k(z),
\end{equation}
where $\ell_j^k(z) = \prod\limits_{i=0,i\neq j}^{k}\frac{(z-z_i)}{(z_j-z_i)}$
are Lagrange polynomials of degree $k$ associated with $\{z_0,\ldots, z_k\}$,
and
\[
    \mathcal{P}_k([-1,1];V) := \Big\{g \in C^0([-1,1];V) : g(z) = \sum_{j=0}^{k} g_j \ell_j^k(z) \text{ with } g_0,\ldots,g_k \in V \Big\}.
\]

We now establish an estimate for the interpolation error for a function
$g:[-1,1] \to V$ which admits an analytic extension in the complex plane,
recalling a result from \cite[Theorem
8.2]{Trefethen2013}.
\begin{lemma}
    \label{lemma:univariate_convergence_rate}
    Given a function $g \in C^0([-1,1];V)$ that admits an analytic extension in
    the region 
    $\Sigma([-1,1];\tau):=\{z \in \C: \text{dist}(z,[-1,1])\leq \tau
    \}$ for some $\tau > 0$, then its Chebyshev interpolants satisfy
    \begin{equation}
        \label{eq:1d_cheby_inter_error}
        \norm{g - I_k(g)}{C^0([-1,1];V)} \leq \frac{4}{\rho - 1} e^{-k \log(\rho)} \max_{z \in \Sigma([-1,1];\tau)}\norm{g(z)}{V},
    \end{equation}
    where 
    \[
        1 < \rho := \tau + \sqrt{1 + \tau^2}.
    \]
\end{lemma}

We now extend these univariate interpolation operators and the corresponding
error estimates to the multidimensional setting. To this end, and using the
convention $I_{-1} = 0$, we introduce the univariate difference operators
\begin{equation}
    \label{eq:univariate_difference_operator}
    \Delta_k := I_k - I_{k-1}, \text{ for } k\geq 0.
\end{equation}
With these notations, for any $k\geq 0$, we write
\begin{equation}
    \label{eq:interpolation_in_difference_form}
    I_k = \sum_{j=0}^{k}\Delta_j.
\end{equation}
Next, for any multi-index $\nu \in \mathbb{N}^\dglob$, we associate the
multivariate point 
\[
    \para_\nu = (z_{\nu_1}^1, \ldots, z_{\nu_\dglob}^\dglob) \in [-1,1]^\dglob,
\]
and the tensorised multivariate operators
\begin{equation}
    \label{eq:multivariate_operators}
    \cI_\nu := \bigotimes_{j=1}^{\dglob} I_{\nu_j}, \quad \Delta_\nu := \bigotimes_{j=1}^{\dglob} \Delta_{\nu_j}.
\end{equation}
Similarly, the multivariate interpolation operators can be written as the sums
of difference operators, namely, for every $\nu \in \mathbb{N}^{\dglob}$
\begin{equation}
    \label{eq:multi_interpolation_in_difference_form}
    \cI_\nu = \bigotimes_{j=1}^{\dglob} \Big(\sum_{k=0}^{\nu_j} \Delta_k\Big) = \sum_{\mu \leq \nu}\Delta_\mu.
\end{equation}
However, $\cI_\nu$ is the interpolation operator associated with the full tensor
product grid, which involves $\prod_{j=1}^{\dglob}(\nu_j+1)$ degrees of freedom.
When the number of variables $\dglob$ is moderately large, the computation of
$\cI_\nu$ becomes prohibitively expensive due to the curse of dimensionality.
This motivates the use of anisotropic sparse grids, with the aim of
substantially reducing the number of required function evaluations while
preserving the same approximation error, as shown in
\cite{nobileAnisotropicSparseGrid2008}.

We now introduce the anisotropic Smolyak formulas $\cI_{w,\balpha}$,
parameterized by the level $w \in \mathbb{N}$ and the weight vector $\balpha =
(\alpha_1,\ldots,\alpha_{\dglob}) \in \mathbb{R}^{\dglob}$ with $\balpha > 0$.
Moreover, we denote by $\underline{\alpha} := \min_{1\leq i \leq
  \dglob} \alpha_i$ and
by $|\balpha| := \sum\nolimits_{1\leq i \leq \dglob} \alpha_i$, and define the set
\begin{equation}
    \label{eq:anisotropic_set}
    \Lambda_{w,\balpha} := \Big\{ \bw \in \mathbb{N}^{\dglob}: \sum_{i=1}^{\dglob}\bw_i \alpha_i \leq w \underline{\alpha} \Big\}.
\end{equation}
The anisotropic Smolyak formula $\cI_{w,\balpha}$ is then given by
\begin{equation}
    \label{eq:anisotropic_Smolyak_formula}
    \cI_{w,\balpha} = \sum_{\bw \in \Lambda_{w,\balpha}} \Delta_\bw.
\end{equation}
In addition, to compute $\cI_{w,\balpha}(\cdot)$, it suffices to evaluate the
function on the anisotropic sparse grid
\begin{equation}
    \label{eq:anisotropic_sparse_grid}
    \mathcal{S}_{w,\balpha} := \bigcup_{\bw \in \Lambda_{w,\balpha} \backslash \Lambda_{w-|\balpha|/\underline{\alpha}, \balpha}} \Big( \{z_1^1,\ldots,z_{k(\bw_1)}^1\} \times \cdots \times \{z_1^{\dglob},\ldots,z_{k(\bw_{\dglob})}^{\dglob}\} \Big),
\end{equation}
where $k(\bw_i)$ denotes the number of grid nodes in direction $i$ dependent on
$\bw_i$.

The Smolyak formula yields an interpolation operator whenever a nested sequence
of univariate interpolation nodes is used. In what follows, we focus on the
Clenshaw-Curtis nodes in the construction of $\cI_{w,\balpha}$. For any $k_i\in
\mathbb{N}$ with $k_i>1$, these nodes in direction $i$ are given by
\begin{equation}
    \label{eq:Clenshaw-Curtis_nodes}
    z_j^i = -\cos \Big(\frac{\pi(j-1)}{k_i-1}\Big), \quad j = 1,\ldots,k_i.
\end{equation}
Moreover, for a given set $\Lambda_{w,\balpha}$ and for each $\bw \in
\Lambda_{w,\balpha}$, we let the number of nodes in direction $i$ be
\begin{equation}
    k_i = k(\bw_i) = 2^{\bw_i} + 1, \quad \bw_i \geq 1,
\end{equation}
with the convention that $k(0) = 1$ and $z_1^i =0$ when $\bw_i = 0$. With this
choice, we obtain a nested sequence of nodes, and consequently
$\mathcal{S}_{w,\balpha} \subset \mathcal{S}_{w+1,\balpha}$.

The key idea of anisotropic sparse interpolation is to use more nodes in
directions with slower convergence. This property of
$\cI_{w,\balpha}$ is encoded in the choice of the weight vector
$\balpha$. It is essential here, since the
parameter dependence of  the diffusion coefficient $A\left(x, \para\right)$ in
\eqref{eq:parametric_pde} is strongly anisotropic; some of the parameters $z_j$
influence the solution $\u(x,\para)$ more than others. 
Taking all components of the vector $\balpha$ to be the same yields the
isotropic Smolyak formula, which treats all variables equally. 

Hereafter, we assume that $g \in C^0([-1,1]^{\dglob};V)$ admits an analytic
extension to the complex region 
\begin{equation}
    \label{eq:complex_region}
    E(\boldsymbol{\tau})
    := \{\mathbf z \in \mathbb{C}^{\dglob} :
    \operatorname{dist}(z_i,[-1,1]) \le \tau_i,\; i=1,\dots,\dglob\},
\end{equation}
where $\tau_i > 0$ for all $i=1,\ldots,\dglob$. According to \cite[Remark
3.12]{nobileAnisotropicSparseGrid2008}, the choice
\begin{equation}
    \label{eq:weights}
    \alpha_i = \frac{1}{2} \log(\tau_i + \sqrt{1+\tau_i^2})
\end{equation}
is optimal with respect to the error estimate established in
\cite[Theorem~3.8]{nobileAnisotropicSparseGrid2008}. The result is recalled in
the following lemma.
\begin{lemma}
    \label{lemma:multivariate_convergence_rate}
    Let $g \in C^0([-1,1]^{\dglob};V)$ admit an analytic
    extension to the complex region \eqref{eq:complex_region} with $\tau_i >0$ for
    all $i=1,\ldots,\dglob$, and consider the anisotropic sparse interpolator
    $\cI_{w,\balpha}$ based on Clenshaw-Curtis nodes
    \eqref{eq:Clenshaw-Curtis_nodes} with weights $\balpha$ as defined in
    \eqref{eq:weights}. Let $\eta = \eta(w,\dglob) := \#
    \mathcal{S}_{w,\balpha}$ be the number of
    interpolation nodes. Then, there exists a constant $C_{\balpha}$
    independent of $\eta$ such that
    \begin{equation}
      \label{eq:multivariate_rate}
      \left\| g - \cI_{w,\balpha}(g)
      \right\|_{L^\infty([-1,1]^{\dglob};V)} \leq 
      \left\{
        \begin{array}{ll}
          C_{\balpha} \eta^{\rho (\frac12 - \log(2)\mathrm{e})},
          & \text{if} \ \ 0 \leq w \leq \frac{|\balpha|}{\underline{\alpha} \log(2)},\\[1.5ex]
         C_{\balpha} \eta^{\frac{\rho}2} \exp\Big( -|\balpha| \,
          \eta^{\frac{\log(2)}{|\balpha|}\rho} \Big), & \text{if} \
                                                           \ w >
                                                           \frac{|\balpha|}{\underline{\alpha}
                                                           \log(2)},
            \end{array}\right.                                               
    \end{equation}
    where
    \[
     \rho :=
    \frac{\underline{\alpha}}{\log(2) +
      \sum_{i=1}^{\dglob}\underline{\alpha}/\alpha_i} \, .
    \]   
\end{lemma}

The algebraic estimate applies in the preasymptotic regime, while the
subexponential estimate requires the refinement level $w$ to exceed the
threshold stated in the lemma. This threshold is at least $\dglob/\log(2)$, so
reaching the subexponential regime can require very fine grids when many
parameters are present. Consequently, the algebraic estimate is often more
relevant in practical computations
\cite[Remark~3.10]{nobileAnisotropicSparseGrid2008}. For the local interpolation
problems considered below, this condition involves the number of active local
parameters $\dloc$ rather than the global dimension.

\subsection{Grassmannian interpolation error}
Our goal is to approximate the mapping $\specProjR$ via interpolation on the
Grassmann manifold. In the finite-dimensional setting for Riemannian manifolds,
it is shown in \cite{jacobsson2025approximating}, that this can be done by
interpolating the composition of the Riemannian logarithm and $\specProjR$.
First, we state some properties of the Grassmannian manifold. 
Let $M = \GrR_{\nloc}(V)$ with $V$ a real Hilbert space.
We will now express the Grassmannian distance in terms of the principal angles between
subspaces. This will be useful for bounding the local approximation error of the spaces $\AharmEx_{\para}(\specProjRApprox(\para))$ in the GFEM.
Let $V_1, V_2 \in \GrR_k(V)$
and $\theta_1, \ldots, \theta_k$ be the principal angles between $V_1$ and $V_2$
which are defined by
\begin{equation*}
    \cos(\theta_i) = \sigma_i, \quad i=1, \ldots, k,
\end{equation*}
where $\sigma_1 \geq \cdots \geq \sigma_k$ are the singular values of
the matrix with entries $( v_{1,i}, v_{2,j})_V$ for $i,j=1, \ldots, k$
where $\{v_{1,1}, \ldots, v_{1,k}\}$ and $\{v_{2,1}, \ldots, v_{2,k}\}$ are orthonormal
bases of $V_1$ and $V_2$, respectively.
\begin{lemma}
    \label{lem:grassmannian_distance_principal_angles}
    Let $V$ be a Hilbert space and $V_1, V_2 \in M = \GrR_k(V)$. Let $\theta_1, \ldots, \theta_k$ 
    be the principal angles between $V_1$ and $V_2$.
    Then, we have
    \begin{equation}
        \label{eq:grassmannian_distance_principal_angles}
        d_{M}(V_1, V_2) = \sqrt{\theta_1^2 + \cdots + \theta_k^2}.
    \end{equation}
\end{lemma}
\begin{proof}
    For $\R^n$ with the standard Euclidean inner product, 
    this is shown in \cite[p.211]{absil2004riemannian}. Now let $V$ be an
    $n$-dimensional Hilbert space with a general inner product $(\cdot, \cdot)_V$ and 
    $V_1, V_2 \in \GrR_k(V)$.
    There exists a linear isometry $L:(V, (\cdot, \cdot)_V) \to (\R^n, (\cdot, \cdot)_{2})$.
    Then $L$ induces an isometry $\widetilde{L}: \GrR_k(V) 
    \to \GrR_k(\R^n)$ by
    \cite[Lemma 2.3]{biliotti2014properly}. In particular, we obtain 
    \begin{equation}
        \label{eq:dist_grass_general_inner_product}
        d_{\GrR_k(V)}(V_1, V_2) 
        = d_{\GrR_k(\R^n)}(\widetilde{L}(V_1), \widetilde{L}(V_2)) 
        = \sqrt{\sum_{i=1}^k \theta_i^2(\widetilde{L}(V_1), \widetilde{L}(V_2))},
    \end{equation}
    where  $\theta_{i}(\widetilde{L}(V_1), \widetilde{L}(V_2))$ are the principal angles
    between $\widetilde{L}(V_1)$ and $\widetilde{L}(V_2)$ with respect to the
    standard Euclidean inner product.  The first of the equalities is a
    consequence of $\widetilde{L}$ being an isometry and the second equality 
    is the result in the Euclidean setting. 

    We show that the principal angles are invariant under isometries. For the
    spaces $V_1$ and $V_2$ we choose orthonormal bases $\{v_{1,1}, \ldots,
    v_{1,k}\}$ and $\{v_{2,1}, \ldots, v_{2,k}\}$. Then, $\{L(v_{1,1}), \ldots,
    L(v_{1,k})\}$ and $\{L(v_{2,1}), \ldots, L(v_{2,k})\}$ are orthonormal bases
    of $\widetilde{L}(V_1)$ and $\widetilde{L}(V_2)$, respectively, because $L$
    is an isometry. Thus, the matrices with entries $(v_{1,i}, v_{2,j})_V$ and
    $(L(v_{1,i}), L(v_{2,j}))_{\R^n}$ are the same, which implies that the
    principal angles $\theta_i(V_1, V_2)$ and $\theta_i(\widetilde{L}(V_1),
    \widetilde{L}(V_2))$ are the same. Together
    with~\eqref{eq:dist_grass_general_inner_product} this shows the result for
    $V$ finite-dimensional. The infinite-dimensional case now follows from the
    finite-dimensional case using \cite[Theorem 6]{harms2012geodesics}. Let $V$
    be the infinite-dimensional Hilbert space and $V_1, V_2 \in \GrR_k(V)$. By
    \cite[Theorem 6 (5)]{harms2012geodesics} there exists a $2k$-dimensional
    subspace $U \subset V$ such that $V_1$ and $V_2$ are contained in $\GrR_k(U)
    \subset \GrR_k(V)$. 
    
    Note that the principal angles between $V_1$ and $V_2$
    in $\GrR_k(U)$ are the same as the principal angles between $V_1$ and
    $V_2$ in $\GrR_k(V)$. By \cite[Theorem 6(2)]{harms2012geodesics} and
    since we have already shown
    the result for the finite-dimensional case, 
    \begin{equation*}
        d_{\GrR_k(V)}(V_1, V_2) 
        = d_{\GrR_k(U)}(V_1, V_2) 
        = \sqrt{\sum_{i=1}^k \theta_{i}^2(V_1, V_2)}. \qedhere
    \end{equation*}
\end{proof}
Let us fix $p \in M$. We denote the tangent space of $M$ at $p$ as $T_p M$. For
each $v \in T_p M$ let $\gamma_v$ be the unique geodesic with $\gamma_v(0) = p$ and
$\dot{\gamma}_v(0) = v$. As $\GrR_{\nloc}(V)$ is geodesically complete
\cite[Proposition~2.4]{biliotti2014properly}, the Riemannian exponential
$\exp_p \colon T_p M \to M$, $\exp_p(v) := \gamma_v(1)$, is defined on all of
$T_p M$. A neighborhood $U \subset M$ of $p$ is called a normal neighborhood of $p$
if it is the diffeomorphic
image under $\exp_p$ of a star-shaped neighborhood $S \subset T_p M$ of 
$0\in T_p M$.
The Riemannian logarithm $\log_p : U \to S$ is the inverse of $\exp_p$ on $U$.
\begin{proposition}
    \label{prop:properties_grassmannian}
    Let $V$ be a real Hilbert space and $\GrR_{k}(V)$ denote the Grassmann manifold.
    Let $p \in \GrR_{k}(V)$. The Riemannian exponential of $\GrR_{k}(V)$ is real 
    analytic on $T_p \GrR_{k}(V)$. Let $U$ be a normal neighborhood of $p$ and 
    $S$ the corresponding star-shaped neighborhood of $0 \in T_p \GrR_{k}(V)$. Then the Riemannian
    logarithm $\log_p : U \to S$ is real analytic.
\end{proposition}
The proof is given in \cref{proof:properties_grassmannian}.
\begin{proposition}
    \label{prop:exp_diffeo_and_radial_geodesics_minimal}
    Let $V$ be a real Hilbert space and $p\in \GrR_{k}(V)$. 
    The map $\exp_p : B_{\pi/2}(0) \to B_{\pi/2}(p)$ is a 
    diffeomorphism between the open ball $B_{\pi/2}(0) \subset T_p \GrR_{k}(V)$ of radius $\pi/2$ around $0$ 
    and the open ball $B_{\pi/2}(p) \subset \GrR_{k}(V)$ of radius $\pi/2$ around $p$. For all 
    $v\in B_{\pi/2}(0)$, the geodesic $\gamma_v:[0,1] \to \GrR_{k}(V)$ defined by $\gamma_v(t) = \exp_p(tv)$ 
    is the minimal geodesic connecting $p$ and $\exp_p(v)$.
\end{proposition}
\begin{proof}
    To show injectivity, let $v_1, v_2 \in B_{\pi/2}(0)$ be two tangent vectors
    such that $\exp_p(v_1) = \exp_p(v_2)$. Let $\widetilde{p}\in
    \mathrm{St}_k(V)$ be a representative of $p$ and let $\widetilde{v}_1,
    \widetilde{v}_2 \in T_{\widetilde{p}} \mathrm{St}_k(V) \subset
    \linop{\R^k}{V}$ be the horizontal lifts of $v_1$ and $v_2$, respectively.
    We define the subspace 
    $$
        W := \mathrm{span}\{\widetilde{p}(\R^k), \widetilde{v}_1(\R^k),
        \widetilde{v}_2(\R^k)\} \subset V
    $$ and consider the
    inclusion $\iota : W \to V$, which induces isometric immersions $\iota_S :
    \mathrm{St}_k(W) \to \mathrm{St}_k(V)$ and $\iota_G : \GrR_k(W) \to
    \GrR_k(V)$ by \cite[Lemma 2.3]{biliotti2014properly}. We will reduce the 
    problem to the finite-dimensional Grassmannian $\GrR_k(W)$, by providing 
    $p^W \in \GrR_k(W)$ and $v_1^W, v_2^W \in T_{p^W} \GrR_k(W)$ such that
    \begin{equation}
        \label{eq:reduction_to_finite_dimensional_grassmannian}
        d_{p^W}\iota_G(v_i^W) = v_i, 
        \quad \exp_p(t v_i) = \iota_G(\exp^W_{p^W}(t v_i^W)),
        \quad i=1,2.
    \end{equation}
    By construction of $W$, we can regard $\widetilde{p}$, $\widetilde{v}_1$ and
    $\widetilde{v}_2$ as elements of $\linop{\R^k}{W}$, denoted by
    $\widetilde{p}^W$, $\widetilde{v}_1^W$ and $\widetilde{v}_2^W$,
    respectively. Since $\iota$ is isometric the tangency condition is
    preserved, and we have $\widetilde{v}_1^W, \widetilde{v}_2^W \in
    T_{\widetilde{p}^W} \mathrm{St}_k(W)$. 
    Let $\pi^W : \mathrm{St}_k(W) \to \GrR_k(W)$ and 
    $\pi^V : \mathrm{St}_k(V) \to \GrR_k(V)$ be the Riemannian submersions.
    We set 
    $$
    p^W := \pi^W(\widetilde{p}^W), \qquad
    v_i^W := d_{\widetilde{p}^W}\pi^W(\widetilde{v}_i^W)
    $$ 
    and use the identity $\iota_G \circ \pi^W = \pi^V \circ \iota_S$ to infer
    $$
        p
        = \iota_G(p^W)
        ,\quad
        v_i 
        = d_{p^W}\iota_G(v_i^W).
    $$ 
    Moreover, since $\iota_G$ is an isometric immersion, we have
    $$
        \|v_i^W\|_{p^W}
        = \|d_{p^W}\iota_G(v_i^W)\|_{p}
        = \|v_i\|_p < \pi/2.
    $$
    Since $\iota_G(\GrR_k(W))$ is totally geodesic in $\GrR_k(V)$
    \cite[Lemma 2.3]{biliotti2014properly}, the curve
    $t\mapsto \iota_G\big(\exp^W_{p^W}(t\, v_i^W)\big)$ is a geodesic in
    $\GrR_k(V)$ with initial point $p$ and initial velocity $v_i$, so by
    uniqueness of geodesics
    $$
        \exp_p(t\, v_i) = \iota_G\big(\exp^W_{p^W}(t\, v_i^W)\big),
        \qquad t\in[0,1].
    $$
    Let $n = \dim(W) \leq 3k$. By \cite[p.~30]{bendokat2024grassmann}, the
    injectivity radius of $\GrR_k(\R^n)$ is $\pi/2$. Fixing any linear isometry
    between $\R^n$ and $W$, \cite[Lemma 2.3]{biliotti2014properly} yields an
    isometry between $\GrR_k(\R^n)$ and $\GrR_k(W)$, so the injectivity radius
    of $\GrR_k(W)$ is also $\pi/2$. From $\exp_p(v_1) = \exp_p(v_2)$ and
    injectivity of $\iota_G$ we obtain
    $\exp^W_{p^W}(v_1^W) = \exp^W_{p^W}(v_2^W)$, hence $v_1^W = v_2^W$ and
    therefore $v_1 = d_{p^W}\iota_G(v_1^W) = d_{p^W}\iota_G(v_2^W) = v_2$, 
    proving injectivity of $\exp_p$ on $B_{\pi/2}(0)$.

    For surjectivity, consider $q\in B_{\pi/2}(p)$. Then, using \cite[Theorem
    6]{harms2012geodesics} there exists a minimal geodesic $\gamma: [0,1] \to \GrR_k(V)$ of length smaller than $\pi/2$ connecting $p$ and $q$, 
    thus $\exp_p(\dot{\gamma}(0)) = q$ and due to \cite[VIII. Corollary 5.5]{lang2012fundamentals} it holds
    $\dot{\gamma}(0) \in B_{\pi/2}(0)$.
    
    Finally, for the minimality of the geodesic, let $v\in B_{\pi/2}(0)$. By \cite[Theorem 6]{harms2012geodesics}
    there exists a minimal geodesic $\gamma: [0,1] \to \GrR_k(V)$ connecting $p$
    and $\exp_p(v)$. The length of $\gamma$ is given by $\| \dot{\gamma}(0)\|_p$
    and hence $\dot{\gamma}(0) \in B_{\pi/2}(0)$. By injectivity of $\exp_p$ on
    $B_{\pi/2}(0)$, we have $\dot{\gamma}(0) = v$ and hence $\gamma(t) =
    \exp_p(tv)$ is minimal. 
\end{proof}

The interpolation result in \cite{jacobsson2025approximating} is an application
of Toponogov's theorem. In \cref{thm:toponogov} in \cref{sec:toponogov} we show
that it holds for infinite-dimensional Riemannian Hopf-Rinow manifolds as well.
Thus, we get the following proposition.

\begin{proposition}[Variant I]
    \label{prop:interpolation_error_grassmann}
    Let $M := \GrR_{\nloc}(\htraceZR{\os{}})$ and $\optSpace_0^{\partial}\in M$.
    Assume that the image of $\specProjR$ is contained in
    $B_{\pi/2}(\optSpace_0^{\partial})$ and let 
    $\log_{\optSpace_0^{\partial}} : B_{\pi/2}(\optSpace_0^{\partial}) \to B_{\pi/2}(0)$ denote the Riemannian logarithm.
    Then, we have the following results:\\[1ex] 
    (i) The mapping $\specProjRTan = \log_{\optSpace_0^{\partial}} \circ \specProjR: D_{\R}\to B_{\pi/2}(0)$ is real
    analytic and there exists $\boldsymbol{\tau}\in \R_{>0}^{\dloc}$ such that $\specProjRTan$ 
    admits a holomorphic extension to the complex region 
    $E(\boldsymbol{\tau})$ defined in \eqref{eq:complex_region}.\smallskip\\[1ex]
    (ii)
    Let $\specProjRTanApprox := \cI_{w,\balpha}(\specProjRTan|_{\Yloc})$ denote the sparse grid interpolation of
    $\specProjRTan|_{\Yloc}$ and let $C_{\mathrm{I}}$ denote the bound
    on the sparse grid interpolation error from
    Lemma~\ref{lemma:multivariate_convergence_rate}, that is,
    $$
        \| \specProjRTan - \specProjRTanApprox \|_{C^0(\Yloc, T_{\optSpace_0^{\partial}} M)} \leq  C_{\mathrm{I}}.
    $$  
    Furthermore, defining $\specProjRApprox := \exp_{\optSpace_0^{\partial}} \circ \specProjRTanApprox$, 
    we also have
    $$
        d_M(\specProjR(\para), \specProjRApprox(\para)) \leq C_{\mathrm{I}}, \quad \forall \para \in \Yloc. 
    $$
\end{proposition}
\begin{proof}
    (i) Real-analyticity follows from
            \cref{thm:real_analyticity_spectral_projection},
            \cref{prop:exp_diffeo_and_radial_geodesics_minimal} and
            \cref{prop:properties_grassmannian}, since the composition of two
            real analytic functions is again real analytic. 
            
            By \cite[Proposition
            1.2.33]{zech2018sparse}, there exists open $\widetilde{U}\subset
            \C^{\dloc}$ with $D_{\R} \subset \widetilde{U}$ such that
            $\specProjRTan$ admits a holomorphic extension $\widetilde{U}\to
            (T_{\optSpace_0^{\partial}} M)_{\C}$, where
            $(T_{\optSpace_0^{\partial}} M)_{\C}$ is the complexification of
            $T_{\optSpace_0^{\partial}} M$. 
            
            If $\widetilde{U} = \C^{\dloc}$, we
            can take $\tau_i = \infty$. Otherwise, by \cite[Miscellaneous
            Results 6.2, p.~242]{paulmetric} the $\infty$-distance
            $\rho$
            between the compact set $\Yloc$ and the closed set
            $\C^{\dloc}\setminus \widetilde{U}$ is positive. We can thus take
            $\tau_i = \frac{1}{2}\rho$ for all $i=1,\ldots,\dloc$ and
            have $E(\boldsymbol{\tau}) \subset \widetilde{U}$.\medskip

            \noindent
            (ii)    The proof closely follows that of \cite[Theorem
            1]{jacobsson2025approximating} and is an application of Toponogov's theorem, 
            stated in \cref{thm:toponogov} for the convenience of the reader. 
            Let 
            \begin{equation*}
                p = \optSpace_0^{\partial}, \quad v_1 = \specProjRTan(\para)
                \quad \text{and} \quad v_2 = \specProjRTanApprox(\para).
            \end{equation*}
            If one of
            the vectors $v_1$ or $v_2$ is zero, say $v_1=0$, then we have 
            \begin{equation*}
                d_M(\specProjR(\para), \specProjRApprox(\para)) 
                = d_M(p, \exp_p(v_2)) 
                \leq \|v_2\|_p 
                \leq C_{\mathrm{I}}. 
            \end{equation*}
            Thus, we can assume that $v_1, v_2 \neq 0$. Let $l_i = \|v_i\|_p$ for $i=1,2$. 
            Let 
            \begin{align*}
                \gamma_1 : [0, l_1] \to M, \quad t \mapsto \exp_p((l_1-t) \widehat{v}_1), \qquad
                \gamma_2 : [0, l_2] \to M, \quad t \mapsto \exp_p(t \widehat{v}_2),
            \end{align*}
            where $\widehat{v}_i = v_i / \|v_i\|_p$ for $i=1,2$. Due to \cref{prop:exp_diffeo_and_radial_geodesics_minimal},
             $\gamma_1$ is a minimal geodesic. 
             
            We now check the 
            requirements of \cref{thm:toponogov}(B). $M$ is a Hopf-Rinow manifold by 
            \cite[Proposition 2.4]{biliotti2014properly}. By \cite[Corollary 2.1]{biliotti2014properly},
            $M$ has non-negative sectional curvature, thus we can choose $H=0$. 
            Inspection of the proof of \cite[Corollary 2.1]{biliotti2014properly} also gives the upper 
            bound $\Delta=2$, because there the statement is proven by reducing to the finite-dimensional case,
            where the sectional curvature is bounded from above by $2$, see \cite[(4.3)]{bendokat2024grassmann}. 
            
            We choose $M_0$ to be a $2$-dimensional subspace of $T_pM$ containing $v_1$ and $v_2$
            and endow it with the 
            constant Riemannian metric induced by $T_pM$. We define 
            \begin{align*}
                \overline{\gamma_1} : [0, l_1] \to M_0, \quad t \mapsto (l_1-t) \widehat{v}_1, \qquad
                \overline{\gamma_2} : [0, l_2] \to M_0, \quad t \mapsto t \widehat{v}_2. 
            \end{align*}
            These are minimal geodesics in $M_0$ of length $l_1$ and $l_2$, respectively,
            and they define a comparison hinge as in \cref{thm:toponogov}(B).
            Thus all conditions of \cref{thm:toponogov}(B) 
            are satisfied and we obtain
            \begin{equation*}
                d_M(\specProjR(\para), \specProjRApprox(\para)) 
                = d_M(\exp_p(v_1), \exp_p(v_2)) 
                \leq d_{M_0}(\overline{\gamma_1}(0), \overline{\gamma_2}(l_2)) 
                = \|v_1 - v_2\|_p
                \leq C_{\mathrm{I}}.\qedhere
            \end{equation*}
\end{proof}
With the same arguments, we can also prove the following proposition for $\specProjPuR$.
\begin{proposition}[Variant II]
    \label{prop:interpolation_error_grassmann_practical}
    Let
    $M := \GrR_{\nloc}(\Hdir{\os{}})$ be the Grassmannian
    formed with the standard $H^1(\os{})$ inner product. Let $\optSpace_0\in M$.
    Assume that the image of $\specProjPuR$ is contained in
    $B_{\pi/2}(\optSpace_0)$ and let 
    $\log_{\optSpace_0} : B_{\pi/2}(\optSpace_0) \to B_{\pi/2}(0)$ denote the Riemannian logarithm.
    Then, the following results~hold:\\[1ex]
    (i) The mapping $\specProjPuRTan = \log_{\optSpace_0} \circ \specProjPuR: D_{\R}\to B_{\pi/2}(0)$ is real
    analytic and there exists $\boldsymbol{\tau}\in \R_{>0}^{\dloc}$ such that $\specProjPuRTan$ 
    admits a holomorphic extension to the complex region 
    $E(\boldsymbol{\tau})$ defined in \eqref{eq:complex_region}.\\[1ex]
    (ii)
    Let $\specProjPuRTanApprox := \cI_{w,\balpha}(\specProjPuRTan|_{\Yloc})$ denote the sparse grid interpolation of
    $\specProjPuRTan|_{\Yloc}$ and let $C_{\mathrm{II}}$ denote again the bound 
     on the sparse grid interpolation error from
     Lemma~\ref{lemma:multivariate_convergence_rate}, that is,
    $$
        \| \specProjPuRTan - \specProjPuRTanApprox \|_{C^0(\Yloc, T_{\optSpace_0} M)} \leq  C_{\mathrm{II}}.
    $$  
    Furthermore, defining $\specProjPuRApprox := \exp_{\optSpace_0} \circ \specProjPuRTanApprox$, 
    we also have
    $$
        d_M(\specProjPuR(\para), \specProjPuRApprox(\para)) \leq C_{\mathrm{II}}, \quad \forall \para \in \Yloc. 
    $$
\end{proposition}

\begin{remark}
    Our proof in \cref{prop:interpolation_error_grassmann} does not give a quantitative 
    result on the holomorphy region $E(\boldsymbol{\tau})$. This region can be smaller 
    than $\widetilde{D}$. The reason is that the Riemannian logarithm is only real analytic. 
    One could also try to consider everything in the complex setting, including the 
    Grassmannian. The arguments in \cref{thm:real_analyticity_spectral_projection} 
    would then also work to show that 
    $\optSpacePu$ is holomorphic on $\widetilde{D}$. The problem is 
    that the Riemannian logarithm for the complex Grassmannian cannot be holomorphic, see
    \cite[Theorem 10]{bochner1947curvature}, namely as stated in 
    \cite{jentsch2020riemannian} ``Riemannian normal coordinates 
    on K{\"a}hler manifolds are not complex coordinates unless the manifold is flat.'' 
    Thus, the map $\specProjRTan$ would again be only real analytic. This is why we
    restricted directly to the real domain in \cref{thm:real_analyticity_spectral_projection}
    and only have the non-quantitative result on the holomorphy region.  
\end{remark}

\subsection{Local approximation error}
Having established an approximation result for the interpolation of $\specProjR$ and 
$\specProjPuR$ on the Grassmannian, we now analyze the impact on the GFEM error. 
\subsubsection{Variant I}
For this variant, the local approximation spaces are given by the
$A(\para)$-harmonic extensions of the interpolation of $\specProjRApprox$ of
$\specProjR$ in $\GrR_{\nloc}(\htraceZR{\os{}})$ as analyzed in Proposition~\ref{prop:interpolation_error_grassmann}. Note that the interpolation is done in
the ambient space $\htraceZR{\os{}}$ and to obtain functions on the whole
subdomain, we have to compute the $A(\para)$-harmonic extensions, requiring the
solution of local problems after the interpolation of the eigenspaces. We define
\begin{equation}
    \label{def:local_approximation_spaces}
    S^{\para}_{\nloc}(\om{}) := 
    \begin{cases}
    \AharmEx_{\para}(\specProjRApprox(\para))|_{\om{}} + \R,
    &\text{if } \partial\os{}\cap\partial\Omega=\emptyset,\\[6pt]
    \AharmEx_{\para}(\specProjRApprox(\para))|_{\om{}},
    &\text{if } \partial\os{}\cap\partial\Omega\neq\emptyset.
    \end{cases}
\end{equation}
We refer to this choice of local approximation space as Variant I.

The quality of this space is measured by how well they approximate functions
in $H_{A(\para),0}(\os{})$ after applying the partition of unity~$\chi$. This
local approximation error can be controlled by the $(\nloc{}+1)$-th nearly
exponentially decaying eigenvalue $\lambda_{\nloc{}+1}(\para)$, capturing the
approximation quality of the optimal eigenspace, and the angle between the
optimal and interpolated eigenspaces, measuring the quality loss due to
interpolation. Given two finite-dimensional subspaces $V, W \subset
\htraceZR{\os{}}$, we define the angle $\beta_{\para}(V, W)$ by
\begin{equation}
    \label{def:angle_between_subspaces}
    \cos^2(\beta_{\para}(V, W)) 
    := \inf\limits_{\v \in V} \frac{a_{\os{}, \para}^{\partial}(\orthProj{W}\v, \v)}{a_{\os{}, \para}^{\partial}(\v, \v)},
\end{equation}
where $\orthProj{W}$ denotes the orthogonal projection onto $W$ with respect to
$a_{\os{}, \para}^{\partial}(\cdot, \cdot)$. We write $\beta(\para) :=
\beta_{\para}(\specProjR(\para), \specProjRApprox(\para))$. The following
theorem bounds the local approximation error using the interpolated eigenspaces
in terms of $\lambda_{\nloc{}+1}(\para)$ and $\beta(\para)$. It improves the
result in \cite{lipton2022angles} by using ideas on oblique projections in 
\cite{adcock2013beyond,xu2002method}.
\begin{theorem}[Variant I]
    \label{thm:local_approximation_error}
    Let the assumptions of \cref{prop:interpolation_error_grassmann} be satisfied.
    Let $\para\in\Yloc$ and
    suppose $s = C_A \sin(C_{\mathrm{I}}) < 1$ and $C_{\mathrm{I}} < \pi/2$, 
    where $C_{\mathrm{I}}$ is the interpolation error bound from
    \cref{prop:interpolation_error_grassmann} and 
    $$
        C_A = \|\sqrt{A_{\para} / A_{\para_0}} \|_{L^\infty(\os{})} \cdot \|\sqrt{A_{\para_0} / A_{\para}} \|_{L^\infty(\os{})}
    $$
    measures the contrast between the coefficients $A_{\para}$ and $A_{\para_0}$.
    Then, with $S_{\nloc}^{\para}(\om{})$ defined in \eqref{def:local_approximation_spaces},
    \begin{align*}
    \inf\limits_{\v\in S_{\nloc}^{\para}(\om{})}
    \| \chi(\u - \v) \|_{a, \para, \om{}}
    &\leq 
    \lambda_{\nloc{}+1}^{1/2}(\para) 
    (\cos \beta(\para))^{-1} \| \u \|_{a, \para, \os{}} \\
    &\leq  
        \lambda_{\nloc{}+1}^{1/2}(\para) 
        \bigl(1 - s^2\bigr)^{-1/2} \| \u \|_{a, \para, \os{}}
    \end{align*}
    holds for all $\u\in H_{A(\para)}(\os{})\cap H^1(\os{};\R)$ with $ \u|_{\partial\os{}\cap\partial\Omega} = 0$.
\end{theorem}
For the proof, which will be given at the end of this section, we need some lemmas.
\begin{remark}
    The constant $C_A$ is always at least one and equals one if and only if
    the ratio $A_{\para}(x)/A_{\para_0}(x)$ is constant in $x$ on $\os{}$.
    It thus does not necessarily depend on the absolute contrast of $A_{\para}$, but on how
    much $A$ changes between $\para_0$ and $\para$.
\end{remark}
Combining this local result with the 
global-by-local estimate in \cref{thm:global_error_gfem} and choosing 
$S^{\para}_{\nloc}(\om{}) = \AharmEx_{\para}(\specProjRApprox(\para))|_{\om{}}$ on each 
subdomain, we obtain a 
global error bound in terms of the interpolation errors $C_{\mathrm{I}}$ and the eigenvalues 
$\lambda_{\nloc{}+1}(\para)$. This assumes that the local particular functions and 
$A(\para)$-harmonic extensions are computed exactly.
\begin{remark}
    \label{rem:practical_method}
    We note that \cref{thm:local_approximation_error} can't be applied if we interpolated the original
    MS-GFEM eigenspaces $\specProjPuR: \Yloc \to \GrR_{\nloc}(\Hdir{\os{}})$ instead of 
    $\specProjR:\Yloc \to \GrR_{\nloc}(\htraceZR{\os{}})$.
    This is because interpolation in $\Hdir{\os{}}$ does not guarantee that at a 
    parameter $\para$, the interpolated space is still a subspace of $H_{A(\para), 0}(\os{})$,
    which is required for the proof of \cref{thm:local_approximation_error}. 
    From a practical perspective, using the local approximation spaces in 
   ~\eqref{def:local_approximation_spaces} requires the operator harmonic extension and therefore that
    local problems are solved for each parameter $\para$ after the interpolation. 
    The method using an interpolation $\specProjPuRApprox$ of
    $\specProjPuR$ for the local approximation spaces and without solving local problems
    for each parameter $\para$ after interpolation will be referred to as Variant II.
\end{remark}
We begin by establishing a bound similar to 
\cite[Theorem~3.8]{lipton2022angles}.
\begin{lemma}
  \label{lem:range_approximation_error_angle}
  For a compact linear operator \(L\in\linop{X}{Y}\) between Hilbert spaces \(X\)
and $Y$, let \(\{\v_i\}_{i\geq 1}\subset X\) be a complete orthonormal
system of eigenvectors of \(L^*L\)
corresponding to eigenvalues $\lambda_1\geq\lambda_2\geq\cdots\geq 0$
enumerated 
with multiplicity. 
Let \(V=\operatorname{span}\{\v_1,\ldots,\v_k\}\) for some
\(k<\dim X\) and consider a finite-dimensional space
    \(W\subset X\) with orthogonal
    projection \(\orthProj{W}\) such that
    \[
        \cos^2\bigl(\beta(V,W)\bigr)
        :=
        \inf_{\v\in V}
        \frac{(\orthProj{W}\v,\v)_X}{(\v,\v)_X}
        >0.
    \]
    Then, for all \(\u\in X\), there exists \(\w\in W\) such that
    \[
        \|\w\|_X
        \leq
        \bigl(\cos\beta\bigr)^{-1}\|\u\|_X
    \qquad
    \text{and}
    \qquad
        \|L(\u-\w)\|_Y
        \leq
        \lambda_{k+1}^{1/2}
        \bigl(\cos\beta\bigr)^{-1}\|\u\|_X.
    \]
\end{lemma}

\begin{proof}
    In order to apply \cite[Theorem 3.4]{adcock2013beyond}, we
    define the space
    $
        U:=\orthProj{W}(V)
    $
    and, using the notation of \cite{adcock2013beyond}, consider
    \[
        \mathcal{W} := \mathcal{W}_{U,V^\perp}\colon X\longrightarrow X,
        \qquad
        \mathcal{W}_{U,V^\perp}(\u)
        :=
        \bigl(\orthProj{V}|_U\bigr)^{-1}\orthProj{V}\u.
    \]
    For this to be well-defined, we need to show that
    $
        \orthProj{V}|_U\colon U\longrightarrow V
    $
    is invertible. 
    Since \(\orthProj{W}\) is an orthogonal projection,
    \((\orthProj{W}\v,\v)_X=\|\orthProj{W}\v\|_X^2\) for all \(\v\in X\), hence
    \begin{equation*}
        \cos\bigl(\beta(V,W)\bigr)
        =
        \inf_{\v\in V}
        \frac{\|\orthProj{W}\v\|_X}{\|\v\|_X}
        >0.
    \end{equation*}
    In particular \(\orthProj{W}|_V\) is injective, so
    \(U=\orthProj{W}(V)\subset W\) satisfies \(\dim U=\dim V=k\).
    Next, we show that
    $
        \orthProj{W}|_V\colon V\longrightarrow U
    $
    is the adjoint of \(\orthProj{V}|_U\). For \(\u\in U\), \(\v\in V\), it holds
    \[
        (\orthProj{V}|_U\u,\v)_X
        =
        (\u,\v)_X
        =
        (\u,\orthProj{W}|_V \v)_X,
    \]
    since \(U\subset W\). Denoting the smallest singular value of an
    operator by \(\sigma_{\min}(\cdot)\), we conclude
    \begin{equation}
        \label{eq:orth_proj_singular_value}
        \sigma_{\min}(\orthProj{V}|_U)
        =
        \sigma_{\min}(\orthProj{W}|_V)
        =
        \inf_{\v\in V}
        \frac{\|\orthProj{W}\v\|_X}{\|\v\|_X}
        =
        \cos\bigl(\beta(V,W)\bigr)
        >0.
    \end{equation}
    Hence
    \(\orthProj{V}|_U\) is bijective.
    It follows that \(\mathcal{W}\) is an oblique
    projection, that is
    \begin{align*}
        \mathcal{W}^2
        =
        \bigl(\orthProj{V}|_U\bigr)^{-1}\orthProj{V}
        \bigl(\orthProj{V}|_U\bigr)^{-1}\orthProj{V} 
        =
        \bigl(\orthProj{V}|_U\bigr)^{-1}\orthProj{V}
        =
        \mathcal{W}.
    \end{align*}
    Moreover, we have
    $
        \ker\mathcal{W} 
        =
        \ker \orthProj{V}
        =
        V^\perp 
    $
    and     
    $    
        \operatorname{range}\mathcal{W}
        =
        \operatorname{range}\bigl(\orthProj{V}|_U\bigr)^{-1}
        =
        U.
    $
    Furthermore, the identity
    $
        \orthProj{V}\mathcal{W}=\orthProj{V},
    $
    implies
    $
        \operatorname{range}
        \bigl(\operatorname{id}_X-\mathcal{W}\bigr)
        \subset V^\perp.
    $
    The spaces \(U\) and \(V^\perp\) are closed, because \(U\) is
    finite-dimensional and \(V^\perp\) is an orthogonal complement. Moreover, we
    have $X = U \oplus V^\perp$, because every \(\u\in X\) decomposes as
    \(\u=\mathcal{W}\u+(\operatorname{id}_X-\mathcal{W})\u\in U+V^\perp\) and
    the sum is direct, since for every \(\u\in U\cap V^\perp\) there exists
    \(\v\in V\) such that \(\u=\orthProj{W}\v\) and  
    $
        (\u,\u)_X
        = (\orthProj{W}\v,\v)_X
        = 0.
    $ 
    Now, let \(\u\in X\) and set
    $
        \w=\mathcal{W}(\u).
    $
    We apply \cite[Theorem
    3.4]{adcock2013beyond} (with $V^{\perp}$ instead of $V$) to conclude
    \begin{align*}
        \|L(\u-\w)\|_Y
        &=
        \bigl\|
            L\bigl(\operatorname{id}_X-\mathcal{W}\bigr)\u
        \bigr\|_Y 
        \leq
        \|L\|_{\linop{V^\perp}{Y}}
        \bigl\|
            \operatorname{id}_X-\mathcal{W}
        \bigr\|_X
        \|\u\|_X \\
        &\leq
        \lambda_{k+1}^{1/2}
        \left(
            \inf_{\v\in U}
            \frac{\|\orthProj{V}\v\|}{\|\v\|}
        \right)^{-1}
        \|\u\|_X 
        =
        \lambda_{k+1}^{1/2}
        \sigma_{\min}(\orthProj{V}|_U)^{-1}
        \|\u\|_X 
        =
        \lambda_{k+1}^{1/2}
        \cos(\beta)^{-1}
        \|\u\|_X.
    \end{align*}
    Finally, using~\eqref{eq:orth_proj_singular_value}, we obtain
    \begin{equation*}
        \|\w\|_X
        \leq
        \bigl\|\bigl(\orthProj{V}|_U\bigr)^{-1}\bigr\|
        \|\orthProj{V}\u\| 
        \leq
        \frac{1}{\sigma_{\min}(\orthProj{V}|_U)}
        \|\u\|_X
        =
        \cos\bigl(\beta(V,W)\bigr)^{-1}\|\u\|_X. \qedhere
    \end{equation*}
\end{proof}
\begin{corollary}
    \label{cor:lipton_sinz_trace_space_coefficient_dependent}
    Let $\para\in\Yloc$ and suppose that $\cos(\beta(\para)) > 0$. Then
    $$
        \sup_{\u^{\partial}\in \htraceZR{\os{}}}\;\inf_{\v^{\partial}\in \specProjRApprox(\para)}
        \frac{\| \pu{\para}(\u^{\partial} - \v^{\partial}) \|_{a, \para, \om{}}}{\| \u^{\partial} \|_{a^{\partial}, \para, \os{}}}
        \leq \lambda_{\nloc{}+1}^{1/2}(\para) (\cos \beta(\para))^{-1}.
    $$
\end{corollary}
It remains to relate the angle $\beta(\para)$ to the Grassmannian 
interpolation error~$C_{\mathrm{I}}$. The angle $\beta(\para) = \beta_{\para}(\specProjR(\para), \specProjRApprox(\para))$ 
is defined with respect to the parameter-dependent inner product 
$a_{\os{}, \para}^{\partial}(\cdot, \cdot)$. 
In contrast, the Grassmannian distance 
$d_M(\specProjR(\para), \specProjRApprox(\para))$ in 
\cref{prop:interpolation_error_grassmann} is measured with respect to the fixed 
reference inner product $a_{\os{}, \para_0}^{\partial}(\cdot, \cdot)$.
The following proposition connects~$\beta(\para)$ to the principal angles 
with respect to the reference inner product.
For $i = 1, \ldots, \nloc{}$ and $\para, \para^{\prime} \in \Yloc$, we denote by
$\theta_{\para^\prime, i}(\para)$ the $i$-th principal angle between 
$\specProjR(\para)$ and $\specProjRApprox(\para)$ with respect to 
$a_{\os{}, \para^\prime}^{\partial}(\cdot, \cdot)$, see \cref{lem:grassmannian_distance_principal_angles}. 
\begin{lemma}
    \label{prop:beta_angle_is_maximal_principle_angle_coefficient_dependent}
    It holds that
    $\beta(\para) = \theta_{\para, \nloc{}}(\para)$.
\end{lemma}
\begin{proof}
    Let $\orthProj{\specProjRApprox(\para)}$ denote the orthogonal projection 
    onto $\specProjRApprox(\para)$ with respect to 
    $a_{\os{}, \para}^{\partial}(\cdot, \cdot)$. 
    By the variational characterization of principal angles 
    \cite[Theorem~4.3]{knyazev2002principal}, the largest principal angle satisfies
    \begin{align*}
        \sin^2(\theta_{\para, \nloc{}}(\para))
        &= \max_{\substack{\u\in \specProjR(\para)\\ \|\u\|_{a^{\partial}, \para, \os{}} = 1}} 
        \| (I - \orthProj{\specProjRApprox(\para)})\u \|_{a^{\partial}, \para, \os{}}^2 
        = \max_{\substack{\u\in \specProjR(\para)\\ \|\u\|_{a^{\partial}, \para, \os{}} = 1}} 
        1 - a_{\os{}, \para}^{\partial}(\orthProj{\specProjRApprox(\para)}\u, \u) \\
        &= 1 - \min_{\substack{\u\in \specProjR(\para)\\ \|\u\|_{a^{\partial}, \para, \os{}} = 1}} 
        a_{\os{}, \para}^{\partial}(\orthProj{\specProjRApprox(\para)}\u, \u) 
        = 1 - \cos^2(\beta(\para)),
    \end{align*}
    where the last equality is the definition of $\beta(\para)$.
\end{proof}
\begin{proof}[Proof of \cref{thm:local_approximation_error}]
The first inequality follows from \cref{cor:lipton_sinz_trace_space_coefficient_dependent},
since every $\u \in H_{A(\para),0}(\os{})$ can be written as 
$\u = \AharmEx_{\para}^{0}(\u^\partial)$ for a unique
$\u^\partial \in \htraceZR{\os{}}$, and similarly every 
$\v \in \AharmEx_{\para}^{0}(\specProjRApprox(\para))$ takes the form
$\v = \AharmEx_{\para}^{0}(\v^\partial)$ with
$\v^\partial \in \specProjRApprox(\para)$. For an interior domain, these
functions belong to $S^{\para}_{\nloc}(\om{})$ by
\eqref{eq:equal_local_spaces_AharmEx_zero} and for a boundary domain this follows
from $\AharmEx_{\para}^{0}=\AharmEx_{\para}$. We note that for interior
subdomains, the first inequality follows for all $\u \in H_{A(\para)}(\os{})$, since the constant
functions are added to the local approximation space and the energy norm of 
constant functions vanishes, cf. the proof of \cite[Theorem 5.4]{alber2025optimal}.

It remains to show $(\cos \beta(\para))^{-1} \leq (1 - C_A^2 \sin^2(C_{\mathrm{I}}))^{-1/2}$.
Writing $M = \GrR_{\nloc{}}(\htraceZR{\os{}})$ and using 
~\eqref{eq:grassmannian_distance_principal_angles} together with 
\cref{prop:interpolation_error_grassmann}, we obtain
\begin{equation}
    \label{proof:principal_angle_bound_by_interpolation_error}
    \theta_{\para_0, \nloc{}}(\para)
    \leq \Bigl( \sum_{i=1}^{\nloc{}} \theta_{\para_0, i}(\para)^2 \Bigr)^{1/2} 
    = d_M(\specProjR(\para), \specProjRApprox(\para)) \leq C_{\mathrm{I}}.
\end{equation}
To change from the reference to the parameter-dependent inner product, we note that the
identity operator $\mathrm{Id}: (\htraceZR{\os{}}, a^{\partial}_{\os{}, \para_0}) 
\to (\htraceZR{\os{}}, a^{\partial}_{\os{}, \para})$ is a linear isomorphism with 
condition number $\kappa(\mathrm{Id}) \leq C_A$.
Indeed, for any $\u^\partial \in \htraceZR{\os{}}$,
\begin{align*}
    a_{\os{}, \para}^{\partial}(\u^\partial, \u^\partial)
    &= a_{\os{}, \para}(\AharmEx_{\para}\u^\partial, \AharmEx_{\para}\u^\partial) 
    = \inf \bigl\{ \|\u\|_{a, \para, \os{}}^2 : \u\in H^1(\os{}),\; \trace(\u) = \u^\partial \bigr\}\\
    &\leq \|A_{\para} / A_{\para_0} \|_{L^\infty(\os{})} 
    \inf \bigl\{ \|\u\|_{a, \para_0, \os{}}^2 : \u\in H^1(\os{}),\; \trace(\u) = \u^\partial \bigr\} \\
    &= \|A_{\para} / A_{\para_0} \|_{L^\infty(\os{})} \, 
    a_{\os{}, \para_0}^{\partial}(\u^\partial, \u^\partial).
\end{align*}
By symmetry, the reverse bound holds with $\|A_{\para_0} / A_{\para}\|_{L^\infty(\os{})}$, 
establishing $\kappa(\mathrm{Id}) \leq C_A$.
Applying \cref{prop:principal_angles_linear_isomorphism} and~\eqref{proof:principal_angle_bound_by_interpolation_error} yields
\begin{equation}
    \label{eq:sin_angle_change}
    \sin(\theta_{\para, \nloc{}}(\para)) 
    \leq C_A \sin(\theta_{\para_0, \nloc{}}(\para))
    \leq C_A \sin(C_{\mathrm{I}}).
\end{equation}
By \cref{prop:beta_angle_is_maximal_principle_angle_coefficient_dependent}, 
$\beta(\para) = \theta_{\para, \nloc{}}(\para)$. From~\eqref{eq:sin_angle_change}, 
\begin{equation*}
    \cos(\beta(\para)) 
    = \sqrt{1 - \sin^2(\beta(\para))}
    \geq \sqrt{1 - (C_A \sin(C_{\mathrm{I}}))^2} 
    > 0,
\end{equation*}
where the last inequality uses the assumption $C_A \sin(C_{\mathrm{I}}) < 1$.
Taking reciprocals yields
\begin{equation*}
    (\cos \beta(\para))^{-1} \leq \bigl(1 - C_A^2 \sin^2(C_{\mathrm{I}})\bigr)^{-1/2}. \qedhere
\end{equation*}
\end{proof}

\subsubsection{Variant II}

We now analyze the method mentioned in \cref{rem:practical_method}, which uses
the interpolation $\specProjPuRApprox$ of the original MS-GFEM eigenspaces
$\specProjPuR$ in $\GrR_{\nloc}(H_D^1(\os{}; \R))$ for the local
approximation spaces as analyzed in \cref{prop:interpolation_error_grassmann_practical}.
In the GFEM, we thus use the local approximation spaces
\begin{equation}
    \label{def:local_approximation_spaces_practical}
    S^{\para}_{\nloc}(\om{}) = 
    \begin{cases}
    \specProjPuRApprox(\para)|_{\om{}} + \R,
    &\text{if } \partial\os{}\cap\partial\Omega=\emptyset,\\[6pt]
    \specProjPuRApprox(\para)|_{\om{}},
    &\text{if } \partial\os{}\cap\partial\Omega\neq\emptyset.
    \end{cases}
\end{equation}
In contrast to Variant I, in Variant II there are no $A(\para)$-harmonic extensions and thus no
local problems need to be solved for each parameter $\para$ after interpolation.
\begin{lemma}
    \label{lem:constraint_poincare}
    There is a constant $C_F>0$, independent of $\para\in\Yloc$, such that
    \begin{equation}
        \label{eq:constraint_poincare}
        \|\u\|_{L^2(\os{})}
        \leq C_F\|\nabla\u\|_{L^2(\os{})},
        \qquad
        \u\in H_{A(\para),0}(\os{})\cap H^1(\os{};\R).
    \end{equation}
    Consequently, for all $\u\in H_{A(\para),0}(\os{})\cap H^1(\os{};\R)$
    \begin{equation}
        \label{eq:H1_energy_norm_equivalence}
        \frac{\Amin^{1/2}}{\sqrt{1+C_F^2}}
        \|\u\|_{H^1(\os{})}
        \leq \|\u\|_{a,\para,\os{}}
        \leq \Amax^{1/2}\|\u\|_{H^1(\os{})}.
    \end{equation}
\end{lemma}
\begin{proof}
    Let the domain be interior and set
    $\overline\u:=|\os{}|^{-1}\int_{\os{}}\u\,\mathrm{d}x$. The constraint
    $a_{\om{},\para}(\chi\u,\chi)=0$ gives
    \begin{equation*}
        |\overline\u|a_{\om{},\para}(\chi,\chi)^{1/2}
        \leq
        \|\chi(\u-\overline\u)\|_{a,\para,\om{}}.
    \end{equation*}
    The Poincar\'{e} inequality for $\u-\overline\u$, the product rule, and
    $a_{\om{},\para}(\chi,\chi)
    \geq\Amin\|\nabla\chi\|_{L^2(\om{})}^2$ imply
    \begin{equation*}
        |\overline\u|
        \leq
        \frac{\Amax^{1/2}
        (\|\chi\|_{L^\infty(\om{})}
        +C_{F,0}\|\nabla\chi\|_{L^\infty(\om{})})}
        {\Amin^{1/2}\|\nabla\chi\|_{L^2(\om{})}}
        \|\nabla\u\|_{L^2(\os{})},
    \end{equation*}
    where $C_{F,0}$ is the Poincar\'{e} constant on functions with zero mean
    on $\os{}$. Hence, in this case,~\eqref{eq:constraint_poincare} follows with
    \begin{equation}
        \label{eq:def_C_F}
        C_F
        :=C_{F,0}+|\os{}|^{1/2}
        \frac{\Amax^{1/2}
        (\|\chi\|_{L^\infty(\om{})}
        +C_{F,0}\|\nabla\chi\|_{L^\infty(\om{})})}
        {\Amin^{1/2}\|\nabla\chi\|_{L^2(\om{})}}.
    \end{equation}
    using $\|\u\|_{L^2(\os{})}\leq\|\u-\overline\u\|_{L^2(\os{})} + |\os{}|^{1/2}|\overline\u|$. 
    For a boundary domain,~\eqref{eq:constraint_poincare} follows from the
    Poincar\'{e} inequality on $\Hdir{\os{}}$, with $C_F$ denoting the
    corresponding Poincar\'{e} constant. The norm equivalence
    \eqref{eq:H1_energy_norm_equivalence} now follows from
    \eqref{eq:constraint_poincare} and uniform ellipticity.
\end{proof}

    The interpolated eigenspaces
    $\specProjPuRApprox(\para)$ are not contained in $H_{A(\para),0}(\os{})$,
    but this is required for
    \cref{lem:range_approximation_error_angle}. We thus orthogonally project the
    interpolated eigenspaces onto $A(\para)$-harmonic functions, apply the
    result for the projected space, and then bound the error introduced by the
    projection. Let $\orthProj{A(\para)}$ denote the $H^1(\os{})$-orthogonal
    projection from $\Hdir{\os{}}$ onto $H_{A(\para),0}(\os{})\cap
    H^1(\os{};\R)$. We fix
    $\para$ and let $W := \specProjPuRApprox(\para)$ and $W^H := \orthProj{A(\para)}
    W$. The angle $\beta^H_{\para}$ is defined by
    \begin{equation*}
    \cos^2(\beta^H_{\para}) 
    := \inf\limits_{\v \in \specProjPuR(\para)} 
    \frac{a_{\os{}, \para}(\orthProj{W^H}^a\v, \v)}{a_{\os{}, \para}(\v, \v)},
    \end{equation*}
    where $\orthProj{W^H}^a$ is the orthogonal projection from
    $\Hdir{\os{}}$ to $W^H$ with respect to
    $a_{\os{}, \para}$. Moreover, let
    \begin{equation}
        \label{eq:def_C_1_A}
        C_{1,A}
        :=\sqrt{\Amax/\Amin}\sqrt{1+C_F^2},
        \qquad
        C_{\chi, A}
        :=C_{1,A}\left(1+\|\nabla\chi\|_{L^\infty(\om{})}\right).
    \end{equation}
\begin{theorem}[Variant II]
    \label{thm:local_approximation_error_practical}
    Let the assumptions of \cref{prop:interpolation_error_grassmann_practical} be satisfied.
    Let $\para\in\Yloc$ and
    suppose $s = C_{1,A} \sin(C_{\mathrm{II}}) < 1$ and $C_{\mathrm{II}} < \pi/2$,
    where $C_{\mathrm{II}}$ is the interpolation error bound from
    \cref{prop:interpolation_error_grassmann_practical}.
    Then, with $S_{\nloc}^{\para}(\om{})$ defined in~\eqref{def:local_approximation_spaces_practical},
    \begin{align*}
    \inf\limits_{\v\in S_{\nloc}^{\para}(\om{})}
    \| \chi(\u - \v) \|_{a, \para, \om{}}
    &\leq (\cos(\beta^H_{\para}))^{-1}
    \left(\lambda_{\nloc{}+1}^{1/2}(\para)
    + C_{\chi, A} \tan(\theta_{\para})\right) \| \u \|_{a, \para, \os{}}\\
    &\leq
    \Bigl[
    \lambda_{\nloc{}+1}^{1/2}(\para)
    \bigl(1 - s^2\bigr)^{-1/2}
    + C_{\chi, A}\, \frac{s}{1 - s^2}
    \Bigr] \| \u \|_{a, \para, \os{}}
    \end{align*}
    holds for all $\u\in H_{A(\para)}(\os{})\cap H^1(\os{};\R)$ with $ \u|_{\partial\os{}\cap\partial\Omega} = 0$.
    Here $\theta_{\para}$ denotes
    the largest principal angle between $\specProjPuR(\para)$ and $W$ with
    respect to the $H^1(\os{})$ inner product.
\end{theorem}

\begin{proof}
    For interior subdomains, it suffices to prove the bound for $\u \in
    H_{A(\para),0}(\os{})$. Indeed, any $\u \in H_{A(\para)}(\os{})$ can be
    written as $\u = \u_0 + \mathcal{M}_{\para}(\u)$ with $\u_0 := \u -
    \mathcal{M}_{\para}(\u) \in H_{A(\para),0}(\os{})$, the constants are
    contained in $S^{\para}_{\nloc}(\om{})$, and the energy norm of constants
    vanishes. Let $\u \in H_{A(\para), 0}(\os{})$.
    By construction, we
    have $W^H\subset H_{A(\para), 0}(\os{})$. Thus, if $\cos{\beta^H_{\para}} >
    0$, we can use \cref{lem:range_approximation_error_angle} to infer the
    existence of $\w^H\in W^H$ such that
    \begin{equation}
        \label{eq:practical_method_local_approximation_error_projection}
        \| \chi(\u - \w^H) \|_{a, \para, \om{}}
        \leq
        \lambda_{\nloc{}+1}^{1/2}(\para)
        (\cos(\beta^H_{\para}))^{-1}
         \| \u \|_{a, \para, \os{}}
     \text{  and  } 
        \|\w^H\|_{a, \para, \os{}} \leq (\cos(\beta^H_{\para}))^{-1} \|\u\|_{a, \para, \os{}}. 
    \end{equation}
    Thus we now show that $\cos{\beta^H_{\para}} > 0$. Then
    \cref{lem:grassmannian_distance_principal_angles,prop:interpolation_error_grassmann_practical}
    give
    \begin{equation*}
        \sin(\theta_{\para}) \leq \sin(C_{\mathrm{II}}) \leq s.
    \end{equation*}
    We use that $\specProjPuR(\para) \subset H_{A(\para), 0}(\os{})$ and \cite[Theorem 4.3]{knyazev2002principal} 
    to infer
    \begin{equation}
        \label{eq:norm_w_minus_orth_proj}
        \|(I - \orthProj{A(\para)}) \w \|_{H^1(\os{})}
        \leq \|(I - \orthProj{\specProjPuR(\para)}) \w \|_{H^1(\os{})}
        \leq \sin(\theta_{\para}) \|\w \|_{H^1(\os{})},
        \quad \forall \w\in W.  
    \end{equation}
    Using orthogonality, we conclude 
    \begin{equation}
        \label{eq:lower_bound_orth_proj}
        \|\orthProj{A(\para)} \w \|_{H^1(\os{})}
        = \sqrt{\|\w \|^2_{H^1(\os{})} - \|(I - \orthProj{A(\para)}) \w \|^2_{H^1(\os{})}}
        \geq \cos(\theta_{\para})\|\w \|_{H^1(\os{})},
        \quad \forall \w\in W.  
    \end{equation}
    This shows that $\orthProj{A(\para)}|_{W}$ is injective and thus
    $\dim W^H=\dim W = \nloc$. Let $\theta_{\para}^H$ be the largest principal
    angle between $\specProjPuR(\para)$ and $W^H$ with respect to the
    $H^1(\os{})$ inner product. We have
    \begin{align*}
        \sin^2(\theta_{\para}^H)
        &= \max_{\substack{\v\in \specProjPuR(\para)\\ \|\v\|_{H^1(\os{})} = 1}}
        \| (I - \orthProj{W^H})\v \|_{H^1(\os{})}^2
        = \max_{\substack{\v\in \specProjPuR(\para)\\ \|\v\|_{H^1(\os{})} = 1}}
        \min_{\substack{\w^H\in W^H}} \| \v - \w^H \|_{H^1(\os{})}^2\\
        &= \max_{\substack{\v\in \specProjPuR(\para)\\ \|\v\|_{H^1(\os{})} = 1}}
        \min_{\substack{\w\in W}} \| \v - \orthProj{A(\para)}(\w) \|_{H^1(\os{})}^2
        \leq \max_{\substack{\v\in \specProjPuR(\para)\\ \|\v\|_{H^1(\os{})} = 1}}
        \| \orthProj{A(\para)}(\v) - \orthProj{A(\para)}(\orthProj{W}(\v)) \|_{H^1(\os{})}^2\\
        &\leq \max_{\substack{\v\in \specProjPuR(\para)\\ \|\v\|_{H^1(\os{})} = 1}}
        \| \v - \orthProj{W}(\v) \|_{H^1(\os{})}^2
        = \sin^2(\theta_{\para}) \leq \sin^2(C_{\mathrm{II}}),
    \end{align*}
    the last equality follows from \cite[Theorem 4.3]{knyazev2002principal}.
    By the same variational characterization, $\beta^H_{\para}$ is the
    largest principal angle between $\specProjPuR(\para)$ and $W^H$ with
    respect to the energy inner product. Hence, by
    \cref{prop:principal_angles_linear_isomorphism} and
    \eqref{eq:H1_energy_norm_equivalence}, whose two inequalities give a
    condition number bounded by $C_{1,A}$,
    \begin{equation*}
        \sin(\beta^H_{\para}) \leq C_{1,A}\sin(\theta_{\para}^H)
        \leq C_{1,A}\sin(C_{\mathrm{II}}) = s.
    \end{equation*}
    Thus, we have
    \begin{equation}
        \label{eq:lower_bound_cos_beta}
        \cos(\beta^H_{\para}) \geq \sqrt{1 - s^2} > 0
    \end{equation}
    and there exists a $\w^H \in W^H$ that satisfies~\eqref{eq:practical_method_local_approximation_error_projection}. 
    Let $\w \in W$ be such that $\w^H = \orthProj{A(\para)} \w$. Then, by the triangle inequality 
    \begin{equation}
        \label{eq:triangle_inequality_practical_method}
        \| \chi(\u - \w) \|_{a, \para, \om{}}
        \leq \| \chi(\u - \w^H) \|_{a, \para, \om{}}
        + \| \chi(\w^H - \w) \|_{a, \para, \om{}}
    \end{equation}
    and we only have to bound the second term. By the product rule,
   ~\eqref{eq:norm_w_minus_orth_proj},~\eqref{eq:lower_bound_orth_proj},
   ~\eqref{eq:H1_energy_norm_equivalence}, and
   ~\eqref{eq:practical_method_local_approximation_error_projection},
    we have 
    \begin{align*}
        \| \chi(\w^H - \w) \|_{a, \para, \om{}}
        &\leq \Amax^{1/2}
        \left(1 + \|\nabla\chi\|_{L^\infty(\om{})}\right)
        \|(I - \orthProj{A(\para)})\w\|_{H^1(\os{})}\\
        &\leq \Amax^{1/2}
        \left(1 + \|\nabla\chi\|_{L^\infty(\om{})}\right)
        \sin(\theta_{\para}) \|\w\|_{H^1(\os{})}\\
        &\leq \Amax^{1/2}
        \left(1 + \|\nabla\chi\|_{L^\infty(\om{})}\right)
        \tan(\theta_{\para}) \|\w^H\|_{H^1(\os{})}\\
        &\leq C_{\chi, A} \tan(\theta_{\para}) \|\w^H\|_{a,\para,\os{}}\\
        &\leq C_{\chi, A} \tan(\theta_{\para})
        (\cos(\beta^H_{\para}))^{-1} \|\u\|_{a,\para,\os{}}\\
        &\leq C_{\chi, A} \frac{s}{1-s^2} \|\u\|_{a,\para,\os{}}.
    \end{align*}
    Inserting this and~\eqref{eq:practical_method_local_approximation_error_projection} 
    into~\eqref{eq:triangle_inequality_practical_method} and using~\eqref{eq:lower_bound_cos_beta} yields the desired result.
\end{proof}

\begin{remark}
    The theoretical error bound for Variant I of the method in
    \cref{thm:local_approximation_error} is more robust than the bound for
    Variant II in \cref{thm:local_approximation_error_practical} in the
    following sense. Assume we have a sequence of integers $\nloc_k \to \infty$
    such that \cref{thm:local_approximation_error} and
    \cref{thm:local_approximation_error_practical} hold for each $\nloc_k$ with
    constants $C_{\mathrm{I},k}$ and $C_{\mathrm{II},k}$, and set
    $s_{\mathrm{I},k} = C_A \sin(C_{\mathrm{I},k})$ and $s_{\mathrm{II},k} =
    C_{1,A} \sin(C_{\mathrm{II},k})$. For Variant I, let the interpolation be
    chosen such that $s_{\mathrm{I},k} \leq 1/2$ for all $k$. Then, the bound in
    \cref{thm:local_approximation_error} is 
    \begin{align}
        \label{eq:local_approximation_error_variant_I_robustness}
    \inf\limits_{\v\in S_{\nloc_k}^{\para}(\om{})}
    \| \chi(\u - \v) \|_{a, \para, \om{}}
    \leq  
        \lambda_{\nloc_k+1}^{1/2}(\para) 
        2/\sqrt{3} \| \u \|_{a, \para, \os{}},
    \end{align}
    that is, for fixed bounded interpolation error, the approximation error is
    bounded by the usual approximation error of the optimal eigenspace times a
    constant independent of the interpolation error. In contrast, for Variant
    II, the additional term $C_{\chi, A}\, s_{\mathrm{II},k} / (1 -
    s_{\mathrm{II},k}^2)$ in the bound of
    \cref{thm:local_approximation_error_practical} is not multiplied by the
    eigenvalue $\lambda_{\nloc_k+1}^{1/2}(\para)$ and thus this bound does not
    yield an estimate
    like~\eqref{eq:local_approximation_error_variant_I_robustness}. To achieve a
    bound for Variant II
    like~\eqref{eq:local_approximation_error_variant_I_robustness}, one would
    have to choose the interpolation error $s_{\mathrm{II},k}$ such that it
    balances the decay of the eigenvalue $\lambda_{\nloc_k+1}^{1/2}(\para)$,
    that is $s_{\mathrm{II},k} \sim \lambda_{\nloc_k+1}^{1/2}(\para)\lesssim
    \exp(-c \nloc_k^{1/\spatDim})$ and hence $s_{\mathrm{II},k}$ would have to
    decay very rapidly to zero as $k \to \infty$.
\end{remark}

\section{Numerical experiments}
\label{sec:numerical_experiments}

In this section, we present numerical experiments to illustrate the performance
of the proposed Grassmannian interpolation method for constructing local
approximation spaces. In all numerical experiments, we consider the spatial
domain $\Omega = [0,1]^2$ with a uniform square mesh with mesh size $h$. We use
bilinear finite elements and partition $\Omega$ into $\numDomPerDir$ subdomains
in each direction for the MS-GFEM. The overlapping and oversampling regions
$\omi$ and $\osi$ are constructed by extending the subdomains by a fixed number of
fine element layers. We construct the sparse-grid interpolant using the Sparse
Grids Matlab package \cite{PiazzolaTamelliniSGK:24}, employing Clenshaw--Curtis
points in $[-1,1]$ as interpolation nodes in all experiments. The test nodes
are specified separately for each experiment below. Note that for the theory, the
Grassmannians for Variant I and Variant II used fixed reference inner products,
$a^{\partial}_{\para_0, \os{}}$ and the $H^1(\os{})$ inner product,
respectively. However, in the numerical experiments, we use the $l^2$ inner
product of nodal coefficient vectors of the boundary trace vectors and the full
local finite-element vectors for Variant I and Variant II, respectively, to form
the Grassmannians. Coefficients are treated as piecewise constant on the fine
mesh. 

To measure the local quality of the interpolated local approximation spaces for
Variant I, we use~\eqref{def:angle_between_subspaces} to introduce the angles
$\beta_k = \beta(\optSpace_{k, \R}^\partial(\para), \specProjRApprox(\para))$
for $k=1,\ldots,\nloc$, measuring how well the space $\optSpace_{k,
\R}^\partial(\para)$, spanned by the first $k$ eigenfunctions, is approximated
by the interpolated local approximation space $\specProjRApprox(\para)$. Thus, a
small $\beta_k$ indicates good approximation of the first $k$ eigenfunctions.
For $k = \nloc$, this is the quantity that appears in the local approximation
error bound for the GFEM solution with interpolated local approximation spaces,
see \cref{thm:local_approximation_error}. For the Variant II, we consider the
relevant angles $\beta^H$ and $\theta_{\para}$ appearing in the error bound of 
\cref{thm:local_approximation_error_practical}

To measure the global quality of the interpolated local approximation spaces, we
construct them on each subdomain and compute the corresponding GFEM solution
$\widehat{\u}^{\mathrm{G}}(\para)$. Note that we construct the particular
functions and solve the global system with a sparse direct solver, to isolate
the error due to the local approximation spaces. This means that the particular
functions are recomputed for each test parameter. The GFEM solution with the
optimal local approximation spaces is denoted by $\u^{\mathrm{G}}(\para)$ and
the reference solution computed on the fine mesh by $\u^{\mathrm{e}}(\para)$.
The relative energy error for $\u(\para)\in\{ \u^{\mathrm{G}}(\para),
\widehat{\u}^{\mathrm{G}}(\para) \}$ is given by
\begin{equation*}
    \mathbf{err} :=
        \frac{\|\u(\para) - \u^{\mathrm{e}}(\para)\|_{a, \para}}{\|\u^{\mathrm{e}}(\para)\|_{a, \para}}.
\end{equation*}
For all experiments the source term is chosen as 
$f(x) = 10^3 \exp(-10((x_1-0.15)^2+(x_2-0.55)^2))$ 
and boundary conditions are homogeneous Dirichlet on the left and right boundaries,
and homogeneous Neumann on the top and bottom boundaries.
\subsection{Eigenvalues and Gaps}
The goal of this section is to confirm the theoretical results on the quality
of the interpolated local approximation spaces.
To this end, we study how the eigenvalues and their gaps vary with the parameter, first
restricting to a single parameter ($\dglob = 1$) with $z = z_1 \in [-1,1]$.

\graphicspath{{figures/}}

\subsubsection{Experiments for one parameter}
In this experiment, we only consider Variant I. We consider a $2\times 2$
checkerboard coefficient with one parameter $\para \in [-1,1]$ that controls the
contrast between the two diagonal pairs of quadrants.
The coefficient is given by
$$
A(x,\para) = A_0(x) + \para A_1(x),
$$
where $A_0(x) = 1$ and
$$
A_1(x) = \begin{cases} \phantom{-}c & \text{if } (x_1 - 0.5)(x_2 - 0.5) > 0, \\ -c & \text{if } (x_1 - 0.5)(x_2 - 0.5) \leq 0. \end{cases}
$$
where $c=199/201$ is chosen such that the contrast $\Amax/\Amin = 200$. The
meshsize is chosen as $h=1/90$. Here we use $\numDomPerDir = 3$ subdomains in
each direction, so that the middle subdomain contains the cross-point of the
checkerboard. Both the oversampling layers and the overlap consist of $2$ fine
element layers. The basepoint for the Grassmann interpolation is chosen as
$\specProjR(-1)$. For interpolation, we use $33$ Clenshaw--Curtis points, while testing uses
a higher-level Clenshaw--Curtis grid. 

This problem poses two challenges. First, at high contrast, the
operator-harmonic space contains functions with very low spatial regularity. For
a coefficient contrast $R$, the checkerboard coefficient admits Kellogg's
singular function \cite{kellogg1972higher} in polar coordinates $(r,\vartheta)$
centered at the cross-point,
$\u_{\mathrm{sing}}(r,\vartheta)=r^\alpha\Phi(\vartheta)$, where
$\alpha=(4/\pi)\arctan(R^{-1/2})$ and $\Phi$ is given in \cite[Eq.~(4.7),
Example~1]{petzoldt2001regularity}. At $\para=1$, the contrast attains its
maximal value $R=\Amax/\Amin=200$, and $\alpha\approx0.09$.
The singular function $\u_{\mathrm{sing}}$ belongs to the operator-harmonic
space $H_{A(1)}(\os{})$, where $\os{}$ is the middle subdomain's oversampling
domain containing the cross-point. Moreover,
$\u_{\mathrm{sing}}\in H^{1+s}(\os{})$ only for $0\le s<\alpha$, while
$\u_{\mathrm{sing}}\notin H^{1+\alpha}(\os{})$; see \cite[p.~436 and
Lemma~3.1]{petzoldt2001regularity}.

Second, when the parameter moves from $-1$ to $1$, the
low-permeability and high-permeability regions switch, causing the
eigenvalues to cross, as shown in \cref{fig:crosspoint1d}~(top left).

We compute the angles numerically for the middle subdomain with $\nloc=7$ and
$\nloc=5$, shown in \cref{fig:crosspoint1d}~(top and bottom, respectively). We
mainly consider the angle for the largest index, $\beta_7$ for $\nloc=7$ and
$\beta_5$ for $\nloc=5$, while the smaller angles are also displayed in grey.
The spectrum over the full parameter range is shown in
\cref{fig:crosspoint1d}~(top left), and the resulting global energy error
between interpolated and optimal local approximation spaces in
\cref{fig:crosspoint1d}~(bottom left). There is a uniform spectral gap between
the $7$th and $8$th eigenvalues across all sampled parameters. Consequently, all
$\beta_k$ for $k \leq \nloc$ remain smaller than $10^{-5}$. However, when $\nloc
= 5$, there is no gap at $\para=0$ between the $5$th and $6$th eigenvalues,
causing several $\beta_k$ to be large, even the smaller indices $k < 5$ exhibit
visible oscillations $>0.1$. For the computation of the MS-GFEM solution with
interpolated local approximation spaces, we construct for each subdomain a
separate interpolated local approximation space of dimension $6$ and $8$
corresponding to the choice $\nloc = 5$ and $\nloc = 7$ on the interior domain,
respectively. The resulting global energy errors are shown in
\cref{fig:crosspoint1d}~(bottom left). Note that the global energy error is
significantly larger for $\nloc = 5$ than for $\nloc = 7$, confirming the
theoretical results on the local approximation error. The curve using the opimal
spaces is visually the same as for the interpolated spaces for $\nloc = 7$, but
not for $\nloc = 5$.

\begin{figure}[t]
    \includegraphics[width=0.45\textwidth]{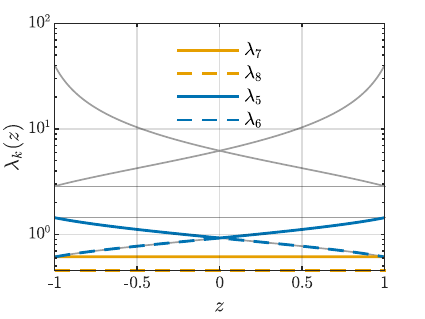}
    \quad
    \includegraphics[width=0.45\textwidth]{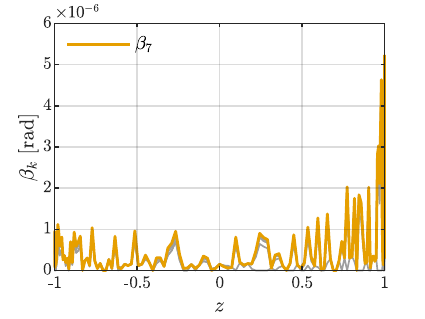}
    \\
    \includegraphics[width=0.45\textwidth]{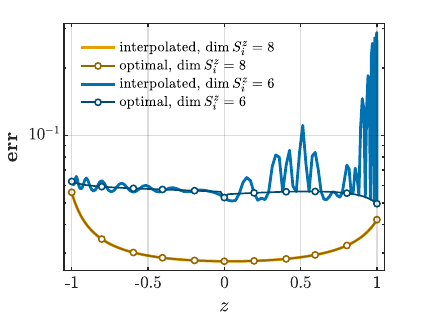}
    \quad
    \includegraphics[width=0.45\textwidth]{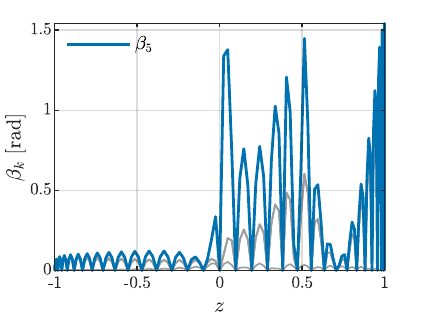}
    \caption{
        Local and global quality of MS-GFEM spaces with interpolated local approximation
        spaces of Variant I for the one-dimensional cross-point problem.
        Top left: spectrum over the parameter range for the middle subdomain.
        Top right: $\beta_7$ angles for $\nloc = 7$ (uniform gap between eigenvalues $7$ and $8$).
        Bottom right: $\beta_5$ angles for $\nloc = 5$ (no uniform gap).
        Smaller angles are shown in grey.
        Bottom left: relative energy error with interpolated MS-GFEM spaces
        and optimal spaces.
    }
    \label{fig:crosspoint1d}
\end{figure}

\subsubsection{Experiments for two parameters}

In this experiment, we investigate Variant II and again consider a checkerboard
coefficient, but now with two parameters $\para = (z_1, z_2) \in [-1,1]^2$
controlling the coefficient in the top right and bottom left quadrants,
respectively. We choose again $h = 1/90$, $\numDomPerDir = 3$,
and both the oversampling layers and the overlap consist of $2$ fine element
layers. The coefficient is given by
\begin{equation*}
    A(x, \para) = A_0(x) + z_1 A_1(x) + z_2 A_2(x),
\end{equation*}
where
\begin{equation*}
    A_0(x) = 1, \quad
    A_1(x) = c \chi_{\{x_1>0.5,\; x_2>0.5\}}(x),\quad
    A_2(x) = c \chi_{\{x_1 < 0.5,\; x_2 < 0.5\}}(x),
\end{equation*}
and $c>0$ is chosen such that the contrast satisfies $\Amax/\Amin = 10000$. We
choose the dimension of the local spaces to be $56$,
for which the MS-GFEM error is very small, so that the observed convergence
reflects only the interpolation error. In \cref{fig:crosspoint2d} (left), we
show the distribution of the relative energy error as a function of the number of
interpolation points. In \cref{fig:crosspoint2d} (right), we show the
corresponding distributions of the Variant-II angles $\beta^H_{\para}$ and
$\theta_{\para}$ for the interior subdomain. At every interpolation level, the local
and global quantities are evaluated on the same common set of off-grid test
parameters, obtained from a sparse grid constructed with equidistant
one-dimensional nodes. Both distributions shift towards smaller values as the number of
interpolation points increases, which is consistent with the theoretical results on the local
approximation error.

\begin{figure}[t]
    \includegraphics[width=0.45\textwidth]{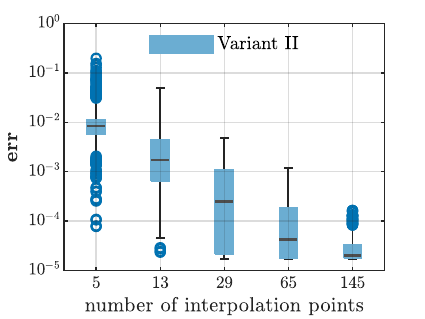}
    \quad
    \includegraphics[width=0.45\textwidth]{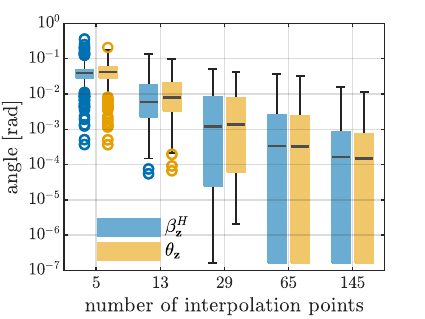}
    \caption{Two-parameter cross-point problem. Left: relative energy error.
    Right: Variant-II angles $\beta^H_{\para}$ and $\theta_{\para}$ for the
    interior subdomain. These statistics are computed after
    the $\log_{10}$ transformation displayed on the vertical axes.}
    \label{fig:crosspoint2d}
\end{figure}

\subsection{Affine case with many parameters}

In this subsection, we use triangular Leja points for testing and Variant II
unless otherwise stated; the final global error experiment compares
Variants I and II.

We consider an $m\times m$ checkerboard partition of $\Omega$, with $m$ even. Let
\begin{equation*}
    \Omega_{ij} := \left[\frac{i-1}{m}, \frac{i}{m}\right] \times \left[\frac{j-1}{m}, \frac{j}{m}\right],
    \qquad i,j\in\{1,\dots,m\},
\end{equation*}
and define the ``active'' cells by
\begin{equation*}
    \mathcal{K}_m := \{ \Omega_{ij} : i+j \text{ is even} \}.
\end{equation*}
Hence, the problem has $m^2/2$ parameters. In the experiments below $m=4$, and choose a fixed enumeration $\mathcal{K}_4 = \{Q_1,\dots,Q_8\}$ of
the eight ``active'' cells. The coefficient is then given by
\begin{equation*}
    A(x,\para) = 1 + c \sum_{j=1}^8 z_j \chi_{Q_j}(x), \qquad
    \para = (z_1,\dots,z_8) \in [-1,1]^8.
\end{equation*}
The coefficient in the remaining cells is fixed to the value $1$. The constant $c>0$
is chosen such that $\Amax/\Amin \approx 20000$.
\Cref{fig:affine_checkerboard_coeff} shows a sample coefficient field for the
case $m=4$. We use $h=1/100$ and a $5\times 5$ subdomain decomposition. The
overlap consists of two layers of fine elements, and the oversampling extension
comprises five layers. With this choice, each local interpolation problem
depends on at most two parameters, although the global parameter space is
eight-dimensional. In MS-GFEM we choose $16$-dimensional local approximation spaces for all subdomains. On subdomains 
not touching the global Dirichlet boundary, these spaces consist of the constant 
function and the first $15$ nonconstant eigenfunctions of the local eigenproblem.
Throughout this subsection, $n$ and $\widehat n$ denote the numbers of
nonconstant eigenfunctions retained in the optimal and interpolated local
spaces, respectively. We determine the dimensions of the interpolated approximation spaces based on the local
spectral gaps. To quantify these gaps numerically, we compute on each subdomain,
\[
    \delta_k(\para) = \frac{\lambda_k(\para) - \lambda_{k+1}(\para)}{\lambda_k(\para)},
    \qquad k=1,\dots,n,
\]
and select the number of nonconstant eigenfunctions for interpolation as
\[
\widehat{n} = \max\{k\in\{1,\dots,n\}: \delta_k(\para) \geq \tau \text{ for all } \para \in \mathcal{Z}_{\mathrm{train}}\},
\]
where $\mathcal{Z}_{\mathrm{train}}$ is a set of interpolation nodes in the parameter domain $[-1,1]^d$. The threshold is empirically set to $\tau = 0.08$ for this experiment. This choice of $\tau$ is not derived from a rigorous theoretical criterion. Establishing a robust general rule for its selection is left for future work.

\begin{figure}[t]
    \includegraphics[trim=0 3bp 0 0,clip,width=0.48\textwidth]{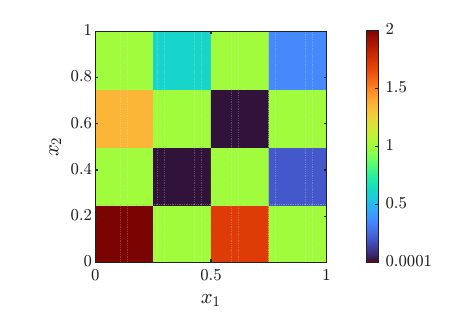}
    \caption{Sample coefficient field for the affine checkerboard problem with
    $m=4$, corresponding to $m^2/2=8$ parameters. The parameter-dependent
    cells $Q_j$ form a checkerboard pattern, while the remaining cells stay
    fixed.}
    \label{fig:affine_checkerboard_coeff}
\end{figure}

Although each local interpolation problem depends on at most two parameters, the spectral behavior can still be very different because the local coefficient geometry differs. We illustrate this by considering the interior subdomains $7$ and $13$, whose local coefficient fields are shown in \Cref{fig:affine_checkerboard_subdomains}. Both subdomains depend on exactly two active parameters. For subdomain $13$, the spectrum exhibits a visible gap between the $15$th and $16$th nonconstant eigenvalues. 
In \cref{fig:affine_checkerboard_gap13}, the values of $\delta_{15}$ shown as a
colormap are uniformly bounded away from zero over the full parameter range, which
indicates a uniform gap between these two eigenvalues. Therefore, in subdomain $13$ we retain $\widehat{n}_{13}=15$ nonconstant eigenfunctions, yielding a local interpolation space of dimension $\widehat{n}_{13}+1=16$ (including the constant function).
By contrast, as shown in \cref{fig:affine_checkerboard_gap7} for subdomain $7$, the values of $\delta_{15}$ are close to zero in a large portion of the parameter range, which means there is no visible uniform gap at the target cutoff. Therefore, in subdomain $7$, we choose the local interpolation dimension adaptively based on the spectral gap criterion.

\begin{figure}[t]
    \includegraphics[trim=0 3bp 0 0,clip,width=0.48\textwidth]{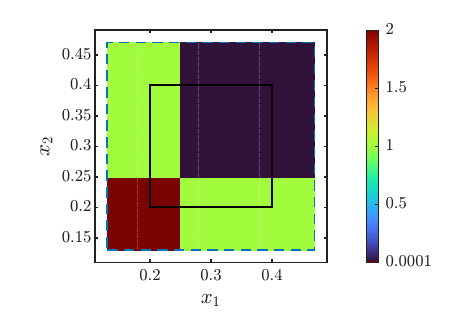}
    \includegraphics[trim=0 3bp 0 0,clip,width=0.48\textwidth]{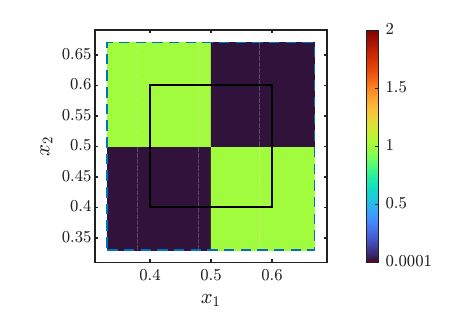}
    \caption{Sample local coefficient fields on the oversampling regions of
    the interior subdomains $7$ (left) and $13$ (right).}
    \label{fig:affine_checkerboard_subdomains}
\end{figure}

\begin{figure}[t]
    \includegraphics[trim=0 3bp 0 0,clip,width=0.9\textwidth]{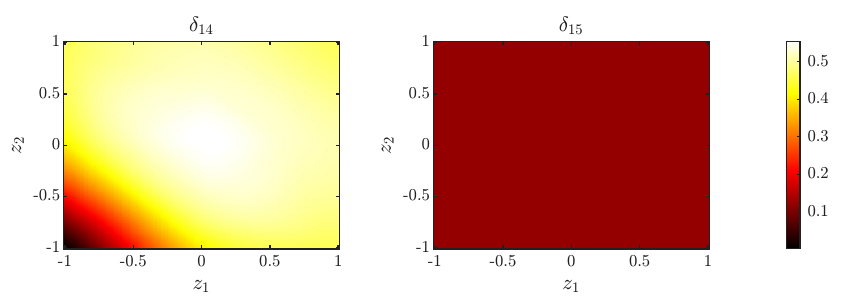}
    \caption{Relative spectral gaps $\delta_{14}$ (left) and $\delta_{15}$
    (right) for subdomain $13$, displayed with a shared color scale.}
    \label{fig:affine_checkerboard_gap13}
\end{figure}

\begin{figure}[t]
    \includegraphics[trim=0 3bp 0 0,clip,width=0.9\textwidth]{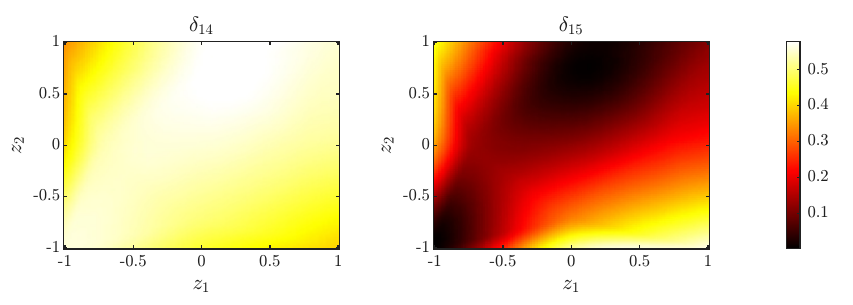}
    \caption{Relative spectral gaps $\delta_{14}$ (left) and $\delta_{15}$
    (right) for subdomain $7$, displayed with a shared color scale.}
    \label{fig:affine_checkerboard_gap7}
\end{figure}

This difference is reflected in the local interpolation error, shown in
\cref{fig:affine_checkerboard_local}.
The left plot uses the adaptive cutoff selected for subdomain $7$ from the
gap criterion, whereas the right plot uses $\widehat\nloc = 15$ on the same subdomain, giving a total local dimension of $16$ including the constant function.
Under the adaptive strategy, the boxplots of the $\beta_{\widehat\nloc}$ angles show a clear decay as the number of interpolation points increases. In contrast, the non-adaptive choice of $\widehat\nloc = 15$ yields large $\beta_{\widehat\nloc}$ angles with no visible decay, confirming that the absence of a uniform spectral gap at the target cutoff leads to poor approximation quality of the interpolated local approximation spaces. 

\begin{figure}[t]
    \includegraphics[trim=0 3bp 0 0,clip,width=0.45\textwidth]{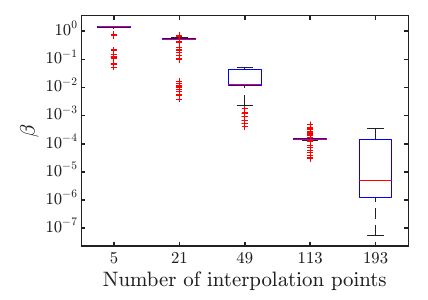}
    \quad
    \includegraphics[trim=0 3bp 0 0,clip,width=0.45\textwidth]{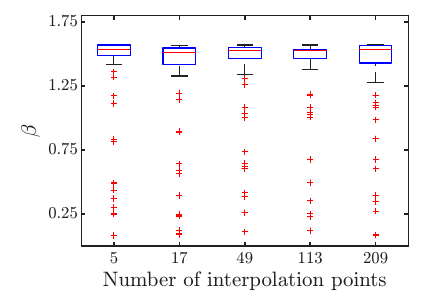}
    \caption{Local interpolation quality for subdomain $7$ using Variant II in both panels.
    Left: adaptive choice of the local interpolation dimension based on the
    spectral-gap criterion.
    Right: non-adaptive choice on the same subdomain.
    The comparison shows that using the correct number of local eigenfunctions
    is essential for obtaining good interpolation quality.}
    \label{fig:affine_checkerboard_local}
\end{figure}

Finally, we determine the dimension of the interpolated local approximation space adaptively on each subdomain and compute the resulting global energy errors, shown in \cref{fig:affine_checkerboard_global} for Variant II (left) and Variant I (right). As the number of interpolation points increases, the error decreases steadily and approaches the error obtained from the MS-GFEM. In this experiment, Variant I achieves comparable global accuracy with fewer interpolation points than Variant II. Since the adaptive selection strategy retains $\widehat{n} \leq n$ nonconstant eigenfunctions on each subdomain, with the constant function added to both spaces whenever required, it is expected that the resulting error may be larger than the MS-GFEM error even for a large number of interpolation points.

The interpolation nodes in both \cref{fig:affine_checkerboard_local} and
\cref{fig:affine_checkerboard_global} are constructed adaptively in all panels.
The adaptive/non-adaptive distinction in \cref{fig:affine_checkerboard_local}
refers only to the spectral cutoff. Since sparse-grid refinement adds nodes in batches
according to the respective interpolation problem, the prescribed point budget does not determine the final node count exactly. Thus, even with the same budget, the left and right panels may use different numbers of interpolation points.

\begin{figure}[h]
    \includegraphics[trim=0 3bp 0 0,clip,width=0.45\textwidth]{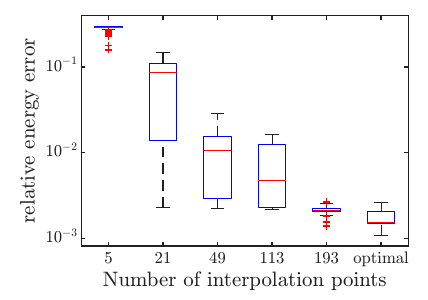}
    \quad
    \includegraphics[trim=0 3bp 0 0,clip,width=0.45\textwidth]{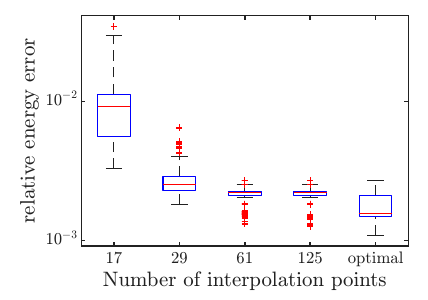}
    \caption{Global energy error for the interpolated local approximation spaces with adaptively determined dimensions: Variant II (left) and Variant I (right). The label ``optimal'' denotes the MS-GFEM solution with the optimal local approximation spaces.}
    \label{fig:affine_checkerboard_global}
\end{figure}

\section{Technical results}
\label{sec:appendix}

\begin{proof}[Proof of \cref{prop:inner_product_for_real}]
    \label{proof:inner_product_for_real}
    We only show this for $\htraceZC{\os{}}$. First, we show that $a_{\os{},\para}^{\partial}(\cdot,
    \cdot)$ is an inner product. Sesquilinearity 
    follows from linearity of $\AharmEx_{\para}$ and sesquilinearity of 
    $a_{\os{},\para}$. 
    Hermiticity follows from the hermiticity of $a_{\os{},\para}$ (because by assumption $A(\para)$ is real-valued) 
    and the fact that
    $\AharmEx_{\compconj{\para}} = \AharmEx_{\para}$ for real $\para$. 
    For positive definiteness, we note that for all $\u^{\partial} \in \htraceZC{\os{}}$
    with $a_{\os{},\para}^{\partial}(\u^{\partial}, \u^{\partial}) = 0$, it follows that
    $\AharmEx_{\para}(\u^{\partial}) = 0$. Taking the trace, we get $\u^{\partial} = 0$. 

    Next, we show completeness. We only consider interior domains where $\htraceZC{\os{}}$ consists of functions with zero mean. 
    The argument for boundary domains is similar.
    Let $(\u_n^{\partial})_{n \in \mathbb{N}}$ be a Cauchy sequence in $\htraceZC{\os{}}$
    with respect to the norm induced by $a_{\os{},\para}^{\partial}(\cdot, \cdot)$. 
    Let us define the linear bounded functional $\mathcal{M}: H^1(\os{};\C)\to \C$ by $\mathcal{M}(\u) = \int_{\partial\os{}} \trace(\u)\,\mathrm{d}s$.
    Then, by definition of the inner product, we have that
    $(\AharmEx_{\para}(\u_n^{\partial}))_{n \in \mathbb{N}}$ is a Cauchy sequence in
    $\mathrm{ker}(\mathcal{M})$ with respect to the energy norm $\|\cdot\|_{a,\para,\os{}}$. 
    By the extended Poincaré–Steklov inequality \cite[Lemma 3.30]{ern2021finite}, there exists a constant $C>0$ such that 
    \begin{equation*}
        \|\u\|_{H^1(\os{})} \leq C \|\u\|_{a,\para,\os{}}, \quad \forall \u \in \mathrm{ker}(\mathcal{M}).
    \end{equation*}
    Therefore, $(\AharmEx_{\para}(\u_n^{\partial}))_{n \in \mathbb{N}}$ is a Cauchy sequence
    in $H^1(\os{};\C)$ and hence converges to some $\u \in H^1(\os{};\C)$. We 
    define $\u^{\partial} := \trace(\u)$ as the trace of $\u$. Since $\u$ is the limit of
    $A(\para)$-harmonic functions, it is also $A(\para)$-harmonic. Thus, $\u$ and
    $\AharmEx_{\para}\u^{\partial}$ are both $A(\para)$-harmonic functions with the same
    trace, which implies that $\u = \AharmEx_{\para}\u^{\partial}$. 
    Moreover, we have $\u^{\partial}\in \htraceZC{\os{}}$, because the trace
    operator is bounded and $(\AharmEx_{\para}(\u_n^{\partial}))_{n \in \mathbb{N}}$
    converges to $\u = \AharmEx_{\para} \u^{\partial}$ in $H^1(\os{};\C)$.
\end{proof}

\subsection{Holomorphy of $T$}
\label{sec:proof_holomorphy_T}
We will several times use the
following product rule for holomorphic operator-valued functions. Let $V,W,X,Y$
be complex Banach spaces and $L:V\to \linop{W}{X}$ and $P:V\to \linop{X}{Y}$ be
holomorphic. Then, by \cite[Proposition 4.11]{zeidler1986_nonlinear}, the
operator $PL:V \to \linop{W}{Y}, (PL)(\v) := P(\v)\circ L(\v)$ is also
holomorphic. 

Since the operator $T$ is defined as the composition of the partition of unity operator
and its Hilbert adjoint, we need to show the holomorphy of both operators. 
\begin{lemma}
\label{prop:partition_of_unity_holomorphic}
    The $A(\para)$-harmonic extension operators $\AharmEx,\AharmEx^{0}: D \to \linop{\htraceZC{\os{}}}{H^1(\os{};\C)}$ and the partition of unity operator
    $\pu: D \to \linop{\htraceZC{\os{}}}{\HzC{\om{}}}$ are
    holomorphic.
\end{lemma}
\begin{proof}
    The idea is to transfer the inhomogeneous boundary condition to the 
    right-hand side. Let $\AharmEx_{\Delta} : \htraceZC{\os{}} \to \H{\os{}}$ 
    denote the harmonic extension operator for the Laplace operator, that is, 
    $\AharmEx_{\Delta}(\u^{\partial})$ is the unique solution of 
    $$
        -\Delta\AharmEx_{\Delta}(\u^{\partial}) = 0, \quad 
        \AharmEx_{\Delta}(\u^{\partial})|_{\partial\os{}} = \u^{\partial}.
    $$
    Writing $\AharmEx_{\para}(\u^{\partial}) = w + \AharmEx_{\Delta}(\u^{\partial})$, 
    $w \in \HzC{\os{}}$ solves the homogeneous boundary value problem
    $$
        -\di(A_{\para} \nabla w) = \di(A_{\para} \nabla \AharmEx_{\Delta}(\u^{\partial})),
        \quad w|_{\partial\os{}} = 0.
    $$
    Due to this representation of $\AharmEx$ as the sum of a solution of a homogeneous boundary value problem 
    and a parameter-independent harmonic extension, we only need to show that $w$ depends holomorphically on $\para$.
    The holomorphy of such solution operators is the content of 
    \cite{cohen2015approximation}. Following \cite[(2.21)]{cohen2015approximation}, we consider the spaces
     $(X,V,W) = (L^{\infty}(\os{}; \C), 
    \HzC{\os{}}, H^{-1}(\os{}; \C))$ for the coefficient, solution, and right-hand side, respectively.
    For $K\in X$ we
    define $\mathcal{B}(K) \in \linop{V}{W}$ by $\mathcal{B}(K)u = -\di( K \nabla u)$. 
    The map $L: X \to \linop{W}{V}$, $K \mapsto \mathcal{B}(K)^{-1}$, is holomorphic 
    at every uniformly elliptic $K \in X$. 
    Define also $Q: X \to \linop{\htraceZC{\os{}}}{W}$ by 
    $Q(K)(\u^{\partial}) = \di(K \nabla \AharmEx_{\Delta}(\u^{\partial}))$, 
    which is linear (hence holomorphic) in $K$. By 
    the product rule, the composition 
    $LQ: X \to \linop{\htraceZC{\os{}}}{V}$ is holomorphic.
    Since $A: D \to X$ is holomorphic and
    $w = (LQ) \circ A(\para)(\u^{\partial})$, we obtain
    $$
        \AharmEx_{\para}(\u^{\partial}) 
        = (LQ) \circ A(\para)(\u^{\partial}) + \AharmEx_{\Delta}(\u^{\partial}),
    $$
    thus $\AharmEx$ is holomorphic in $\para$ as the sum of a holomorphic and a constant map. 

    For an interior domain, uniform ellipticity gives
    \begin{equation*}
        \operatorname{Re}a_{\om{},\para}(\chi,\chi)
        \geq \Amin\|\nabla\chi\|_{L^2(\om{})}^2>0.
    \end{equation*}
    Hence $\para\mapsto a_{\om{},\para}(\chi,\chi)^{-1}$ is holomorphic. The map
    \begin{equation*}
        \para\mapsto
        a_{\om{},\para}(\chi\AharmEx_{\para}\,\cdot,\chi)
        \in\linop{\htraceZC{\os{}}}{\C}
    \end{equation*}
    is also holomorphic, so~\eqref{eq:def_AharmEx_zero} shows that
    $\AharmEx^{0}$ is holomorphic. For a boundary domain this follows from
    $\AharmEx_{\para}^{0}=\AharmEx_{\para}$. 
    
    Finally, the partition of unity
    operator $\pu{\para} = \chi\AharmEx_{\para}^{0}$ is
    holomorphic, because it is the composition of $\AharmEx^{0}$ and
    multiplication by $\chi$, which is a fixed bounded linear operator. 
\end{proof}
From now on, we assume that the domain $D$ is symmetric with respect 
to the real axis. This is not a restriction, since we can always shrink $D$. 
Next, we show that
also the pointwise defined Hilbert adjoint of a holomorphic 
operator-valued map is holomorphic, if the argument is complex-conjugated.
We will need a similar result for the Banach space adjoint, hence 
we treat both in the following lemma.
\begin{lemma}
    \label{lem:antidual_transpose_holomorphic}
    Let $V$ and $W$ be complex Banach spaces with anti-duals $V^*$ and $W^*$.
    For $L \in \linop{V}{W}$, define the Banach adjoint
    $L^{\top} \in \linop{W^*}{V^*}$ by
    \begin{equation*}
        (L^{\top} \phi)(v) := \phi(Lv), \quad \phi \in W^*,\; v \in V.
    \end{equation*}
    Suppose $L: D \to \linop{V}{W}$ is holomorphic. Then, the map
    $D \to \linop{W^*}{V^*}$ given by 
    $\para \mapsto L(\compconj{\para})^{\top}$ is holomorphic.
    If $V$ and $W$ are Hilbert spaces, then the map $D\to\linop{W}{V}$ given by
    $\para \mapsto L(\compconj{\para})^{\ast}$, is holomorphic, where
    $L(\compconj{\para})^{\ast}$ is the Hilbert adjoint of
    $L(\compconj{\para})$.
\end{lemma}

\begin{proof}
    \label{proof:antidual_transpose_holomorphic}
    Since $\phi \in W^*$ is anti-linear, $((\lambda L)^{\top}\phi)(v) = \phi(\lambda Lv)
    = \compconj{\lambda}\,\phi(Lv) = \compconj{\lambda}\,(L^{\top}\phi)(v)$ for all
    $\lambda\in\C$, $L\in\linop{V}{W}$, $\phi\in W^*$, $v\in V$,
    so $L\mapsto L^{\top}$ is $\C$-anti-linear.
    From the definition, $|(L^{\top}\phi)(v)| = |\phi(Lv)| \leq \|\phi\|_{W^*}\|L\|\|v\|_V$,
    giving $\|L^{\top}\|_{\linop{W^*}{V^*}} \leq \|L\|_{\linop{V}{W}}$.
    
    We denote the map given by $\para \mapsto L(\compconj{\para})^{\top}$ by $\widetilde{L}$.
    We claim that $d_{\para}\widetilde{L}(\hb) := (d_{\compconj{\para}}L(\compconj{\hb}))^{\top}$
    is the Fréchet derivative of $\widetilde{L}$ at $\para$,
    where $d_{\compconj{\para}}L$ is the (complex-linear) derivative of $L$ at $\compconj{\para}$.
    For $\alpha\in\C$, using $\C$-linearity of $d_{\compconj{\para}}L$ 
    and $\C$-anti-linearity of taking Banach adjoints, we get
    \begin{equation*}
        d_{\para}\widetilde{L}(\alpha \hb) 
        = \bigl(\compconj{\alpha}\,d_{\compconj{\para}}L(\compconj{\hb})\bigr)^{\top}
        = \alpha\,\bigl(d_{\compconj{\para}}L(\compconj{\hb})\bigr)^{\top}
        = \alpha\,d_{\para}\widetilde{L}(\hb),
    \end{equation*}
    so $d_{\para}\widetilde{L}$ is $\C$-linear, and it is bounded since
    $\|d_{\para}\widetilde{L}(\hb)\|_{\linop{W^*}{V^*}} \leq \|d_{\compconj{\para}}L(\compconj{\hb})\|_{\linop{V}{W}} \leq C\|\hb\|_{\C^{\dloc}}$.
    Using $\|T^{\top}\|_{\linop{W^*}{V^*}}\leq\|T\|_{\linop{V}{W}}$ we obtain
    \begin{align*}
          \|\widetilde{L}(\para+\hb) - \widetilde{L}(\para) - d_{\para}\widetilde{L}(\hb)\|_{\linop{W^*}{V^*}}
          &= \bigl\|\bigl(L(\compconj{\para}+\compconj{\hb}) - L(\compconj{\para}) 
              - d_{\compconj{\para}}L(\compconj{\hb})\bigr)^{\top}\bigr\|_{\linop{W^*}{V^*}} \\
          &\leq \bigl\|L(\compconj{\para}+\compconj{\hb}) - L(\compconj{\para}) 
              - d_{\compconj{\para}}L(\compconj{\hb})\bigr\|_{\linop{V}{W}}
          = o(\|\hb\|_{\C^{\dloc}}),
    \end{align*}
    where the last step uses holomorphy of $L$ at $\compconj{\para}$.
    Hence $\widetilde{L}$ is holomorphic.

    Now let $V$ and $W$ be Hilbert spaces with Riesz isomorphisms
    $R_V\in\linop{V}{V^*}$ and $R_W\in\linop{W}{W^*}$, respectively. Then, a
    direct computation shows that $L(\compconj{\para})^{\ast} = R_V^{-1}
    L(\compconj{\para})^{\top} R_W$. Since $R_V^{-1}$ and $R_W$ are constant in
    $\para$ and hence holomorphic and $\para \mapsto L(\compconj{\para})^{\top}$
    is holomorphic, we conclude that $\para \mapsto L(\compconj{\para})^{\ast}$
    is holomorphic as well.
\end{proof}
Our goal is to show that the operator $T$, initially defined for real 
parameters via~\eqref{eq:def_operator_T}, admits a holomorphic extension to 
a complex domain~$D$. Since the definition of $T_{\para}$ for real $\para$ 
involves the Hilbert adjoint with respect to parameter-dependent inner products, 
which are not available for complex $\para$, we first reformulate $T_{\para}$ 
using a fixed reference inner product. This reformulation will then serve 
as the definition of $T_{\para}$ for complex~$\para$.

\begin{definition}[Operators Encoding the Change of Inner Product]
    \label{def:inner_product_change_operators}
    Let $\para_0 \in \Yloc$ be a fixed reference parameter. 
    For $\para \in D$, the operators $G^{\partial}_{\para}: \htraceZC{\os{}} \to \htraceZC{\os{}}$ and $G_{\para}: \HzC{\om{}} \to \HzC{\om{}}$ are defined by 
    \begin{align*}
        a_{\os{},\para_0}^{\partial}(G^{\partial}_{\para} \u^{\partial}, \v^{\partial}) &= a_{\os{},\para}^{\partial}(\u^{\partial}, \v^{\partial}) \quad \text{for all } \u^{\partial}, \v^{\partial} \in \htraceZC{\os{}}, \\
        a_{\om{},\para_0}(G_{\para} \u, \v) &= a_{\om{},\para}(\u, \v) \quad \text{for all } \u, \v \in \HzC{\om{}}.
    \end{align*}
\end{definition}
These operators are well-defined by the 
    Lax--Milgram theorem \cite[Theorem 2.1]{cohen2015approximation} applied to the fixed reference inner product.
We now establish the holomorphy of $G^{\partial}$ and $G$.

\begin{lemma}
    \label{prop:holomorphy_G_V}
    The operator-valued functions $G^{\partial}:D \to \linop{\htraceZC{\os{}}}$ and $G:
    D\to\linop{\HzC{\om{}}}$ are holomorphic.
    Moreover, there exists an open neighborhood $\widetilde{D} \subset D$ of $\Yloc$ such that
    $G^{\partial}_{\para}$ is invertible for all $\para \in \widetilde{D}$, and
    $(G^{\partial}_{\para})^{-1}$ is also holomorphic on $\widetilde{D}$.
\end{lemma}
\begin{proof}
    We only show this for $G_{\partial}$. The argument for $G$ is simpler, since the
    operator-harmonic extension is not involved in the definition of $G$.
    Define the operator $\mathcal{A}^{\partial}_{\para}: \htraceZC{\os{}} \to (\htraceZC{\os{}})^*$ by 
    \begin{equation*}
        (\mathcal{A}^{\partial}_{\para} \u^{\partial})(\v^{\partial}) 
        = a_{\os{}, \para}^{\partial}(\u^{\partial}, \v^{\partial})
        = a_{\os{}, \para}(\AharmEx_{\para} \u^{\partial}, \AharmEx_{\compconj{\para}} \v^{\partial}).
    \end{equation*}
    Here $(\htraceZC{\os{}})^*$ is the anti-dual of $\htraceZC{\os{}}$ with respect to the fixed reference inner product.
    We decompose this using the Banach space adjoint. Let $\mathcal{A}_{\para}: H^1(\os{}) \to (H^1(\os{}))^*$ 
    be the operator induced by the sesquilinear form $a_{\os{}, \para}$ given by
    \begin{equation*}
        (\mathcal{A}_{\para} \u)(\v) = a_{\os{}, \para}(\u, \v).
    \end{equation*}
    Since $A$ is holomorphic, the operator $\mathcal{A}_{\para}$ is 
    holomorphic in $\para$.
    Let $\AharmEx_{\compconj{\para}}^{\top}: (H^1(\os{}))^* \to (\htraceZC{\os{}})^*$ denote the 
    Banach adjoint of $\AharmEx_{\compconj{\para}}$ (see \cref{lem:antidual_transpose_holomorphic}).
    Applying \cref{lem:antidual_transpose_holomorphic} with $L = \AharmEx$ 
    (which is holomorphic by \cref{prop:partition_of_unity_holomorphic}), 
    we infer that the map $\para \mapsto \AharmEx_{\compconj{\para}}^{\top}$ is holomorphic.
    The operator $\mathcal{A}^{\partial}_{\para}$ can therefore be written as the composition
    $
        \mathcal{A}^{\partial}_{\para} = \AharmEx_{\compconj{\para}}^{\top} \mathcal{A}_{\para} \AharmEx_{\para}
    $
    of holomorphic operator-valued functions, hence $\mathcal{A}^{\partial}_{\para}$ is holomorphic.
    For the invertibility of $\mathcal{A}^{\partial}_{\para}$, we first treat real parameters $\para \in \Yloc$.
    For real $\para$, $\mathcal{A}^{\partial}_{\para}$ is defined via the inner product
    $a^{\partial}_{\os{},\para}(\cdot,  \cdot)$, hence invertibility of $\mathcal{A}^{\partial}_{\para}$ for real $\para$ follows from the complex
    Lax--Milgram theorem \cite[Theorem 2.1]{cohen2015approximation}.
     Since the set of boundedly 
    invertible operators is open in $\linop{\htraceZC{\os{}}}{(\htraceZC{\os{}})^*}$
    (\cite[Remark 1.2.36]{zech2018sparse}) and operator inversion is 
    holomorphic (\cite[Lemma 1.2.35]{zech2018sparse}), we conclude that there 
    exists an open neighborhood $\widetilde{D} \subset D$ of $\Yloc$ such that 
    $\mathcal{A}^{\partial}_{\para}$  is invertible for all $\para \in \widetilde{D}$ and 
    $(\mathcal{A}^{\partial}_{\para})^{-1}$ depends holomorphically on $\para$.
    Since $G^{\partial}_{\para}$ can be expressed as
    \begin{equation*}
        G^{\partial}_{\para} = (\mathcal{A}^{\partial}_{\para_0})^{-1} \mathcal{A}^{\partial}_{\para},
    \end{equation*}
    holomorphy of $G^{\partial}$ follows from the holomorphy of $\mathcal{A}^{\partial}_{\para}$
     and invertibility of $G^{\partial}_{\para}$ for $\para \in \widetilde{D}$  from
     invertibility of $\mathcal{A}^{\partial}_{\para}$.
    Finally, since $(G^{\partial}_{\para})^{-1} = (\mathcal{A}^{\partial}_{\para})^{-1} \mathcal{A}^{\partial}_{\para_0}$, the holomorphy of $(G^{\partial}_{\para})^{-1}$ also follows.
\end{proof}

\begin{lemma}
    \label{lem:adjoint_relation}
    Let $\para \in \Yloc$ and $L: \htraceZC{\os{}} \to \HzC{\om{}}$ be a bounded linear operator. Let $L^{\ast}_{0}$ denote the Hilbert adjoint 
    with respect to the fixed reference inner products $a_{\os{},\para_0}^{\partial}(\cdot, \cdot)$ and $a_{\om{},\para_0}(\cdot, \cdot)$
    , and $L^*$ denote the 
    Hilbert adjoint with respect to the $\para$-dependent inner products
    $a_{\os{},\para}^{\partial}(\cdot, \cdot)$ and $a_{\om{},\para}(\cdot, \cdot)$. Then
    \begin{equation}
        \label{eq:adjoint_relation}
        L^* = (G^{\partial}_{\para})^{-1} L^{\ast}_{0} G_{\para}.
    \end{equation}
\end{lemma}
 
\begin{proof}
    For any $\u^{\partial} \in \htraceZC{\os{}}$ and $\v \in \HzC{\om{}}$, we have 
    \begin{equation}
        \label{eq:adjoint_relation_proof_1}
        a_{\om{},\para}(L \u^{\partial}, \v)
        = a_{\os{},\para}^{\partial}(\u^{\partial}, L^* \v)
        = a_{\os{},\para_0}^{\partial}(G^{\partial}_{\para} \u^{\partial}, L^* \v).
    \end{equation}
    Moreover, using the definition of the operator $G(\para)$, we obtain
    \begin{equation}
        \label{eq:adjoint_relation_proof_2}
        a_{\om{},\para}(L \u^{\partial}, \v)
        = a_{\om{},\para_0}(G_{\para} L \u^{\partial}, \v)
        = a_{\om{},\para_0}(L \u^{\partial}, G_{\para} \v)
        = a_{\os{},\para_0}^{\partial}(\u^{\partial}, L^{\ast}_{0} G_{\para} \v),
    \end{equation}
    where the second equality uses that $G_{\para}$ is self-adjoint with 
    respect to $a_{\om{},\para_0}(\cdot, \cdot)$ for real $\para$, since for 
    $\u,\v \in \HzC{\om{}}$, we have
    \begin{equation}
        \label{eq:adjoint_relation_proof_3}
        a_{\om{},\para_0}(G_{\para} \u, \v)
        = a_{\om{},\para}(\u, \v)
        = \compconj{a_{\om{},\para}(\v, \u)}
        = \compconj{a_{\om{},\para_0}(G_{\para} \v, \u)}
        = a_{\om{},\para_0}(\u, G_{\para} \v).
    \end{equation}
    Combining~\eqref{eq:adjoint_relation_proof_1},~\eqref{eq:adjoint_relation_proof_2} and using that $G^{\partial}_{\para}$ is self-adjoint with respect 
    to $a_{\os{},\para_0}^{\partial}(\cdot, \cdot)$, which follows analogously to~\eqref{eq:adjoint_relation_proof_3}, we infer
    \begin{equation*}
        a_{\os{},\para_0}^{\partial}(G^{\partial}_{\para} \u^{\partial}, L^* \v) = a_{\os{},\para_0}^{\partial}(\u^{\partial}, L^{\ast}_{0} G_{\para} \v) = a_{\os{},\para_0}^{\partial}(G^{\partial}_{\para} \u^{\partial}, (G^{\partial}_{\para})^{-1} L^{\ast}_{0} G_{\para} \v).
    \end{equation*}
    Since this holds for all $\u^{\partial} \in \htraceZC{\os{}}$ and 
    $G^{\partial}_{\para}$ is invertible, we conclude~\eqref{eq:adjoint_relation}.
\end{proof}
 
For real $\para$, applying \cref{lem:adjoint_relation} with $L = \pu{\para}$ to the 
definition~\eqref{eq:def_operator_T} yields
\begin{equation}
    \label{eq:T_reformulated}
    T_{\para} = (G^{\partial}_{\para})^{-1} (\pu{\para})^{\ast}_{0} G_{\para} \pu_{\para},
    \qquad \para \in \Yloc.
\end{equation}
where $(\pu{\para})^{\ast}_{0}$ denotes the Hilbert adjoint of $\pu{\para}$ with 
respect to the fixed reference inner products.
We therefore define the 
holomorphic extension of $T$ by
\begin{equation}
    \label{eq:def_T_complex}
    T_{\para} := (G^{\partial}_{\para})^{-1} (\pu{\compconj{\para}})^{\ast}_{0} G_{\para} \pu_{\para},
    \qquad \para \in \widetilde{D},
\end{equation}
where $\widetilde{D}$ is the neighborhood of $\Yloc$ from \cref{prop:holomorphy_G_V}.
\begin{remark}
    The appearance of $\compconj{\para}$ in the adjoint factor is necessary 
    to ensure holomorphy. The Hilbert adjoint is conjugate-linear, so $\para 
    \mapsto (\pu{\para})^{\ast}_{0}$ is anti-holomorphic. Composing with the 
    conjugation $\para \mapsto \compconj{\para}$ restores holomorphy 
    (see \cref{lem:antidual_transpose_holomorphic}).
\end{remark}

\begin{proof}[Proof of \cref{thm:T_is_holomorphic}]
    By definition~\eqref{eq:def_T_complex}, 
    $T_{\para} = (G^{\partial}_{\para})^{-1} (\pu{\compconj{\para}})^{\ast}_{0} G_{\para} \pu_{\para}$.
    Each factor is holomorphic: $\pu$ by 
    \cref{prop:partition_of_unity_holomorphic}, 
    $G$ by \cref{prop:holomorphy_G_V}, 
    $\para \mapsto (\pu{\compconj{\para}})^{\ast}_{0}$ by 
    \cref{lem:antidual_transpose_holomorphic} (applied with respect to the fixed 
    reference inner products),
    and $\para \mapsto (G^{\partial}_{\para})^{-1}$ by \cref{prop:holomorphy_G_V}.
    Repeated application of the product rule yields the 
    holomorphy of $T$.
    For real $\para$, this definition agrees with~\eqref{eq:def_operator_T} 
    by \cref{lem:adjoint_relation}, so $T$ is indeed a holomorphic extension 
    of the originally defined operator.
\end{proof}
\subsection{Principal angles under linear isomorphisms}
\begin{proposition}
    \label{prop:principal_angles_linear_isomorphism}
    Let $(V, (\cdot, \cdot)_V)$ and $(W, (\cdot, \cdot)_W)$ be two Hilbert spaces. Let $L: V\to W$ be a linear 
    isomorphism. For two $k$-dimensional subspaces $V_1, V_2 \subset V$ the principal angles 
    $\theta_j(V_1, V_2)$ between $V_1$ and $V_2$ and the principal angles
    between $L(V_1)$ and $L(V_2)$, denoted by $\theta_j(L(V_1), L(V_2))$, satisfy
    $$
        \frac{1}{\kappa(L)} \sin(\theta_j(V_1, V_2))
        \leq \sin(\theta_j(L(V_1), L(V_2)))
        \leq \kappa(L) \sin(\theta_j(V_1, V_2)),
    $$
    where $\kappa(L) := \|L\|\ \|L^{-1}\|$ is the condition number of $L$.
\end{proposition}

\begin{proof}
    Since $V_1$ and $V_2$ are contained in a finite-dimensional subspace of 
    $V$, we can assume without loss of generality that $V$ and $W$ are equal to 
    $\R^m$ for some $m\in \mathbb{N}$ with scalar products $(\cdot, \cdot)_V$ and $(\cdot,
    \cdot)_W$ 
    represented by positive definite matrices $M_V$ and $M_W$, respectively. 
    Let $Q_1$ and $Q_2$ be matrices whose columns form orthonormal bases with respect to
    $(\cdot, \cdot)_V$ of $V_1$ and $V_2$, respectively. 
    We use the variational characterization of the sines of the 
    principal angles \cite[Theorem 4.3]{knyazev2002principal} 
    \begin{equation*}
        \sin \theta_j(V_1, V_2)
        = \sigma_j^{\uparrow} (M_V^{1/2}(1 - \orthProj{V_2}) Q_1), 
    \end{equation*}
    where $\sigma_j^{\uparrow}$ denotes the $j$-th (sorted ascending) singular value and $\orthProj{V_i}$ the orthogonal projection onto $V_i$. Using the min-max characterization of 
    the singular values of $(1 - \orthProj{V_2}) \orthProj{V_1}$, we get 
    \begin{align*}
        \sin \theta_j(V_1, V_2)
        = \min_{S\subset V_1, \dim(S) = j} \max_{u\in S, \| u\|_V = 1}
        \| (1 - \orthProj{V_2}) u \|_V
        = \min_{S\subset V_1, \dim(S) = j} \max_{u\in S, \| u\|_V = 1}
        \mathrm{dist}(u, V_2), 
    \end{align*}
    where $\mathrm{dist}(u, V_2) = \min_{v\in V_2} \| u - v\|_V$. Fixing $S\subset V_1$ with
    $\dim(S) = j$,
    $u\in S$ with $\|u\|_V = 1$ and setting $w = L u / \|Lu\|_W$, we have 
    \begin{align*}
        \mathrm{dist}(w, L(V_2))
        &= \min_{v\in V_2} \|\frac{Lu}{\|Lu\|_W} - Lv \|_W 
        = \frac{1}{\|Lu\|_W} \min_{v\in V_2} \| L(u - v) \|_W \\
        &\leq \frac{\|L\|}{\|Lu\|_W} \min_{v\in V_2} \| u - v \|_V
        \leq \|L\|\|L^{-1}\| \min_{v\in V_2} \| u - v \|_V.
    \end{align*}
    Taking the maximum over $u\in S$ and the minimum over $S\subset V_1$ with $\dim(S) = j$, we obtain
    \begin{align*}
        \sin \theta_j(L(V_1), L(V_2))
        &= \min_{S \subset L(V_1), \dim(S) = j} \max_{w\in S, \|w\|_W = 1}
        \mathrm{dist}(w, L(V_2)) \\
        &= \min_{S \subset V_1, \dim(S) = j} \max_{w\in L(S), \|w\|_W = 1}
        \mathrm{dist}(w, L(V_2)) \\
        &= \min_{S \subset V_1, \dim(S) = j} \max_{u\in S, \|u\|_V = 1}
        \mathrm{dist}(\frac{Lu}{\|Lu\|_W}, L(V_2)) \\
        &\leq \|L\|\|L^{-1}\| \min_{S \subset V_1, \dim(S) = j} \max_{u\in S, \|u\| = 1} \mathrm{dist}(u, V_2) \\
        &\leq \|L\|\|L^{-1}\| \sin \theta_j(V_1, V_2).
    \end{align*}
    Applying the same argument with $L^{-1}$ and the spaces $L(V_1)$ and $L(V_2)$, we get 
    \begin{equation*}
        \sin \theta_j(V_1, V_2)
        \leq \|L\|\|L^{-1}\| \sin \theta_j(L(V_1), L(V_2))
    \end{equation*}
    yielding both inequalities.
\end{proof}

\subsection{Analyticity of the exponential map}

\begin{proof}[Proof of \cref{prop:properties_grassmannian}]
    \label{proof:properties_grassmannian}
    To show this, we use the characterisation of the geodesic vector field in 
    \cite[VII, §6, Theorem 3]{Lang_1985} and trace where analyticity has to enter. 
    We let $M := \GrR_k(V)$ and let $E$ denote the model space of $M$, a real
    Hilbert space with inner product $(\cdot, \cdot)_E$. By
    \cref{remark:analytic_manifolds}, $M$ is a real analytic Riemannian
    manifold. Given a chart of $M$ with image $B \subset E$, we
    identify the tangent bundle over the chart domain with $B \times E$
    and, following
    \cite[Chapter~VII, \S 6]{Lang_1985}, represent the Riemannian metric
    $\langle \cdot, \cdot \rangle_x$ by
    the map $g \colon B \to \linop{E}$ with 
    $\langle y, w \rangle_x = (y, g(x) w)_E$. The map $g$ is real analytic by
    \cref{remark:analytic_manifolds}, and hence so is its Fr\'echet
    derivative $dg \colon B \to \linop{E}{\linop{E}}$ by
    \cite[p.~1078]{whittlesey1965analytic}.

    Let us denote by $F$ the geodesic vector field on $TM$. 
    By the last display in the proof of \cite[VII, §6, Theorem~3]{Lang_1985} 
    its local representation $f
    = (f_1, f_2) \colon B \times E \to E \times E$ satisfies $f_1(x,v) = v$
    and
    \begin{equation}
        \label{eq:f2_identity}
        \big( y, g(x) f_2(x,v) \big)_E
        = \tfrac{1}{2} \big( (d_x g(y))v, v \big)_E
        - \big( (d_x g(v))v, y \big)_E
        \qquad \text{for all } y \in E.
    \end{equation}

    Let $T \colon \linop{E}{\linop{E}} \times E \times E \to E$ assign to
    $(A, v, w)$ the Riesz representative of the bounded linear functional
    \begin{equation*}
        y \mapsto \tfrac{1}{2}\big((Ay)v, w\big)_E - \big((Av)w, y\big)_E,
    \end{equation*}
    so that $T$
    is bounded trilinear with $\|T(A,v,w)\|_E \leq \tfrac{3}{2}
    \|A\| \|v\|_E \|w\|_E$, and hence real analytic by
    \cite[I, §3 Proposition~14]{Lang_1985}. By
    \eqref{eq:f2_identity}, $g(x) f_2(x,v) = T(d_xg, v, v) =: \Psi(x,v)$,
    and $\Psi$ is real analytic as a composition of real analytic maps.
    
    To conclude that $f_2$ is real analytic, it remains to show that the map 
    $
        B \to \linop{E}
        , x \mapsto (g(x))^{-1} 
    $
    is real analytic.
    To see this, we apply \cite[Lemma 1.2.25]{zech2018sparse} to the analytic map
    $g$. We fix
    $x\in B$. In this
    lemma we choose $\gamma >0$, where 
    $
        \inf\limits_{v\in E} \frac{(v,g(x)v)_E}{\|v\|^2_E} \geq 1/\gamma.
    $
    Such a $\gamma$ exists since $g$ is a Riemannian metric. 
    Therefore, we have 
    \begin{align*}
        \| (g(x))^{-1} \|_{\linop{E}}
        = \frac{1}{\inf_{v\in E} \frac{\| g(x)(v) \|_{E}}{\| v \|_E}}
        = \frac{1}{\inf_{v\in E}\sup_{w\in E}\frac{|(v, g(x)w)_E|}{\|w\|_E \| v \|_E }} 
            \leq \frac{1}{\inf_{v\in E}\frac{|(v, g(x)v)_E|}{\|v\|_E^2}}
        \leq \gamma. 
    \end{align*}
    Thus, by \cite[Lemma 1.2.25]{zech2018sparse}, the inversion map
    $\mathrm{inv}: \mathcal{B}_{\gamma^{-1}}^{\linop{E}}(g(x)) \to \linop{E}$ is
    analytic, where $\mathcal{B}_{\gamma^{-1}}^{\linop{E}}(g(x))$ denotes the
    open ball of radius $\gamma^{-1}$ centered at $g(x)$ in $\linop{E}$. Now
    let $\widetilde{B}\subset B$ be a neighborhood of $x$ such that we have
    $g(y) \in \mathcal{B}_{\gamma^{-1}/2}^{\linop{E}}(g(x))$ for all $y \in
    \widetilde{B}$. Thus, $\mathrm{inv} \circ g$ is analytic on $\widetilde{B}$
    as the composition of analytic maps. We conclude that $f_2$ is real analytic
    and hence $F$ is real analytic. Now in \cite[IV, §2, Theorem
    5]{Lang_1985}, it is shown that the flow of a $C^{p}$ vector
    field is $C^{p}$, using the Local Smoothness Theorem in IV, §1, which can
    be replaced by the real analytic version \cite[Theorem
    4.D]{zeidler1986_nonlinear} to conclude that the flow $\Theta$ of $F$ is
    real analytic. By \cite[Proposition~2.4]{biliotti2014properly} $M$ is
    geodesically complete. Hence the domain of the flow $\Theta$ contains $\{1\}
    \times TM$. For $v\in T_p M$, the exponential map is given by $\exp_p(v) =
    \pi(\Theta(1,(p,v)))$ with the real analytic projection $\pi \colon TM \to
    M$. Hence $\exp_p$ is real analytic on $T_p M$.
    Real analyticity of the logarithm map follows from real analyticity of the exponential map and the differentiation rule 
    for inverse mappings \cite[Corollary 4.37]{zeidler1986_nonlinear}.
\end{proof}

\subsection{Toponogov's theorem for Hilbert manifolds}
\label{sec:toponogov}

A geodesic triangle in a Hilbert manifold $M$ consists of three unit speed geodesics
$(\gamma_1, \gamma_2, \gamma_3)$ of lengths $(l_1, l_2, l_3)$ 
such that $\gamma_i(l_i) = \gamma_{i+1}(0)$ and $l_i + l_{i+1} > l_{i+2}$ for $i=1,2,3$, where 
the indices are taken modulo $3$. The angle between $-\dot{\gamma}_{i+1}(l_{i+1})$ and 
$\dot{\gamma}_{i+2}(0)$ is denoted by $\alpha_i$.

Let $\gamma_1$ and $\gamma_2$ be two unit speed geodesics in $M$ with lengths 
$l_1$ and $l_2$ such that $\gamma_1(l_1) = \gamma_2(0)$ and $\alpha$ the angle 
between $-\dot{\gamma}_1(l_1)$ and $\dot{\gamma}_2(0)$. We call $(\gamma_1, \gamma_2, \alpha)$
a hinge. 

We call a Hilbert manifold $M$ a Hopf-Rinow manifold if it is complete and 
for every $p,q \in M$ there is at least one minimal geodesic joining 
$p$ and $q$. 
We then have \cite[Theorem 17]{biliotti2006some}: 
\begin{theorem}
    \label{thm:toponogov}
    Let $M$ be a Hopf-Rinow manifold such that its sectional curvature 
    $K^M$ satisfies $H\leq K^M \leq \Delta$ for some $H, \Delta \in \mathbb{R}$.\\
    (A) Let $(\gamma_1, \gamma_2, \gamma_3)$ be a geodesic triangle in $M$. Assume 
    $\gamma_1$ and $\gamma_3$ are minimal geodesics and if $H>0$, $l_2\leq \pi/\sqrt{H}$. 
    Then there exists a geodesic triangle 
    $(\overline{\gamma}_1, \overline{\gamma}_2, \overline{\gamma}_3)$
    in a simply connected $2$-dimensional manifold $M_H$ of constant curvature $H$ such that 
    \begin{equation*}
        l_i = \overline{l}_i, \quad i=1,2,3, \quad \overline{\alpha}_1 \leq \alpha_1, \quad \overline{\alpha}_2 \leq \alpha_2.
    \end{equation*} 
    Except for the case $H>0$ and $l_2=\pi/\sqrt{H}$, 
    the triangle $(\overline{\gamma}_1, \overline{\gamma}_2, \overline{\gamma}_3)$ is
    uniquely determined.\\
    (B) Let $(\gamma_1, \gamma_2, \alpha)$ be a hinge in $M$. 
    Let $\gamma_1$ be a minimal geodesic, and if $H>0$, $l_2\leq \pi/\sqrt{H}$.
    Let $(\overline{\gamma}_1, \overline{\gamma}_2, \overline{\alpha})$ be a hinge in $M_H$ such that 
    $l_i = \overline{l}_i, i=1,2$ and $\alpha = \overline{\alpha}$. Then
    \begin{equation*}
        d_M(\gamma_1(0), \gamma_2(l_2)) \leq d_{M_H}(\overline{\gamma}_1(0), \overline{\gamma}_2(l_2)).
    \end{equation*}
\end{theorem}

\printbibliography
\end{document}